%% file: PlanarDB_submit.tex
\documentclass[reqno,11pt]{amsart}

\usepackage{amsmath,amssymb,amsthm,mathrsfs}
\usepackage{epsfig,color}
\usepackage[latin1]{inputenc}

\numberwithin{equation}{section}

\def\H{\mathcal H}
\def\M{\mathcal M}
\def\R{\mathbb R}
\def\N{\mathbb N}

\def\dist{\hbox{dist}}
\def\e{\varepsilon}
\def\s{\sigma}
\def\S{\Sigma}
\def\vphi{\varphi}

\def\om{\omega}
\def\l{\lambda}
\def\g{\gamma}
\def\k{\kappa}

\def\de{\delta}
\def\Id{{\rm Id}}

\def\spt{{\rm spt}}
\newcommand{\hd}{\mathrm{hd}}

\def\pa{\partial}

\def\ttau{\boldsymbol{\tau}}

\def\E{\mathcal{E}}

\def\F{\mathcal{F}}

\newcommand{\vol}{\mathrm{vol}\,}

\renewcommand{\a}{\alpha}
\renewcommand{\b}{\beta}
\renewcommand{\d}{\mathrm{d}}

\renewcommand{\l}{\lambda}

\renewcommand{\O}{{\mathrm O}}

\newcommand{\ov}{\overline}

\newcommand{\cc}{\subset\subset}

\newcommand{\tE}{\widetilde{\mathcal E}}

\def\aa{\mathbf{a}}

\def\bd{{\rm bd}\,}
\def\INT{{\rm int}\,}

\newtheorem{theorem}{Theorem}[section]
\newtheorem{lemma}[theorem]{Lemma}

\newtheorem{remark}{Remark}[section]

\definecolor{grey}{rgb}{.7,.7,.7}

\title[Stability for planar double-bubbles]{Sharp stability inequalities for planar double bubbles}
\author{M. Cicalese}
\address{Department of Mathematics, Technische Universit\"at M\"unchen, Boltzmannstrasse 3, 85747 Garching, GERMANY}
\email{cicalese@ma.tum.de}
\author{G. P. Leonardi}
\address{Dipartimento di Scienze Fisiche, Matematiche e Informatiche, Universit{\`a} di Modena e Reggio Emilia, Via Campi 213/b, I-41100
    Modena, ITALY}
\email{gianpaolo.leonardi@unimore.it}
\author{F. Maggi}
\address{Department of Mathematics, University of Texas at Austin, Austin, TX, USA}
\email{maggi@math.utexas.edu}

\begin{document}

\begin{abstract}
In this paper we address the global stability problem for double-bubbles in the plane. This is accomplished by combining the {\it improved convergence theorem} for planar clusters developed in \cite{CiLeMaIC1} with an ad hoc analysis of the problem, which addresses the delicate interaction between the (possible) dislocation of singularities and the multiple-volumes constraint.
\end{abstract}

\maketitle


\section{Introduction} The double-bubble theorem in $\R^3$ \cite{HutchMorganRosRitore} asserts that the total perimeter of two regions bounding given volumes is minimized by {\it standard double-bubbles}, which are the familiar soap bubble configurations where three spherical caps meet at 120 degree angles along a circle; see Figure \ref{fig sdb}. A mathematical formulation of this result in the context of finite perimeter sets is given as follows. One says that a family $\E=\{\E(h)\}_{h=1}^N$ of sets of locally finite perimeter in $\R^n$ is a {\it $N$-cluster} in $\R^n$ if $|\E(h)|>0$ for $h=1,...,N$ and $|\E(h)\cap\E(k)|=0$ for $1\le h<k\le N$. We use the term {\it double-bubble} in place of {\it $2$-cluster}. Setting $\E(0)=\R^n\setminus\bigcup_{h=1}^N\E(h)$ for the exterior chamber of $\E$, one defines the perimeter and the volume of $\E$ as
\[
P(\E)=\frac12\sum_{h=0}^NP(\E(h))\,,\qquad\vol(\E)=(|\E(1)|,...,|\E(N)|)\,,
\]
where $P(E)$ and $|E|$ denote, respectively, the distributional perimeter and the Lebesgue measure of a Lebesgue-measurable set $E\subset\R^n$. (In this way, $P(E)=\H^{n-1}(\pa E)$ whenever $E$ is an open set with Lipschitz boundary in $\R^n$, where $\H^k$ is the $k$-dimensional Hausdorff measure on $\R^n$).

For every $m_2\ge m_1>0$, there exists a unique way (up to isometries) to enclose volumes $m_1$ and $m_2$ in $\R^n$ by three $(n-1)$-dimensional spherical caps meeting at 120 degrees angles along a $(n-2)$-dimensional sphere. The corresponding shape is called the standard double-bubble in $\R^n$ (with volumes $m_1$ and $m_2$) and provides the only minimizer (up to isometries) in the isoperimetric problem
\begin{equation}
  \label{isoperimetric problem with two chambers}
  \inf\big\{P(\E):\vol(\E)=(m_1,m_2)\big\}\,,\qquad m_2\ge m_1>0\,,
\end{equation}
as shown in \cite{small_doublebubble} when $n=2$, in \cite{HutchMorganRosRitore} when $n=3$, and in \cite{Reichardt} when $n\ge 4$.
\begin{figure}
  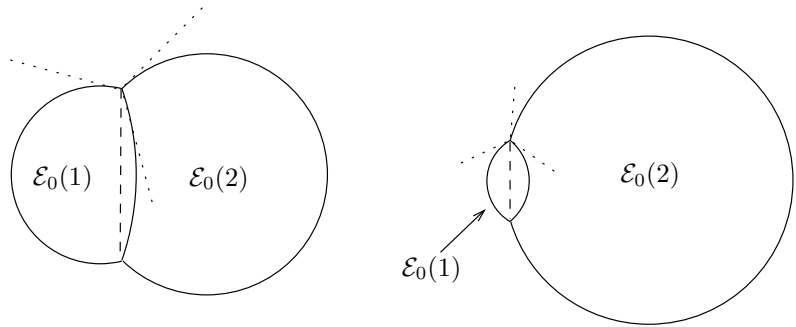\caption{{\small Standard double-bubbles: three $(n-1)$-dimensional spherical caps meeting at 120 degrees angles along a $(n-2)$-dimensional sphere (depicted by a dashed line).}}\label{fig sdb}
\end{figure}
In other words, if $\E_0$ denotes a generic reference standard double-bubble in $\R^n$, then
\begin{equation}
  \label{isoperimetric inequality with two chambers}
  P(\E)\ge P(\E_0)\,,\qquad\mbox{for every double-bubble $\E$ with $\vol(\E)=\vol(\E_0)$}\,,
\end{equation}
with equality if and only if $\E=\E_0$ modulo isometries. Our goal here is, in the planar case $n=2$, to strengthen this isoperimetric inequality in two directions. Our first result is the following sharp quantitative form of \eqref{isoperimetric inequality with two chambers}:

\begin{theorem}[Global stability inequalities]\label{thm main 2}
  If $m_2\ge m_1>0$, then there exists $\k>0$ depending on $m_1$ and $m_2$ only such that, if $\E$ is a planar double-bubble with $\vol(\E)=\vol(\E_0)=(m_1,m_2)$, then, up to isometries,
  \begin{equation}\label{global stability double-bubble}
  P(\E)\ge P(\E_0)\Big\{1+\k\,\Big( |\E(1)\Delta \E_0(1)|+|\E(2)\Delta \E_0(2)|\Big)^2\Big\}\,.
  \end{equation}
\end{theorem}

\begin{remark}\label{remark sharp}
  {\rm We stress the global character of \eqref{global stability double-bubble}, that is to say, $\E$ does not need to be a small perturbation of $\E_0$, or to be parameterized on $\E_0$ in any sense. Moreover, the decay rate in \eqref{global stability double-bubble} is sharp: if $\vphi:[0,\infty)\to[0,\infty)$ is such that $P(\E)\ge P(\E_0)(1+\vphi(\sum_{i=1}^2|\E(i)\Delta\E_0(i)|))$ for every planar double-bubble $\E$ with $\vol(\E)=\vol(\E_0)=(m_1,m_2)$, then there exist $C\ge 0$ and $t_0>0$ such that $\vphi(t)\le C\,t^2$ for every $t\le t_0$; see the discussion before Theorem \ref{thm reduction} below.}
\end{remark}


The typical situation in which we expect to observe double-bubbles $\E$ whose perimeter is close to that of a standard double-bubble $\E_0$ with $\vol(\E_0)=\vol(\E)$, is when $\E$ is the solution to a geometric variational problem sufficiently close to \eqref{isoperimetric problem with two chambers}, like
\begin{equation}
  \label{isoperimetric problem with two chambers plus potential}
  \inf\Big\{P(\E)+\b\,\int_{\E(1)\cup\E(2)}J(x)\,dx:\vol(\E)=(m_1,m_2)\Big\}\,,\qquad\mbox{$\beta>0$ small}\,,
\end{equation}
where $J$ is the density of some potential energy (see also \cite{renwei} for an account on the interaction between the cluster perimeter and a nonlocal repulsive potential). Of course one expects such minimizers to be close to standard double bubbles in a much stronger sense than the one expressed in \eqref{global stability double-bubble}, and we obtain such a quantitative estimate in the following theorem.

\begin{theorem}[Perturbed minimizing clusters]
  \label{thm main perturbed}
  If $m_2\ge m_1>0$ and $J:\R^2\to\R$ is a continuous function with $J(x)\to\infty$ as $|x|\to\infty$, then there exist $C_0>0$ and $\b_0>0$, depending on $m_1$, $m_2$, and $J$ only, with the following property. If $\E_\b$ is a minimizer in the variational problem \eqref{isoperimetric problem with two chambers plus potential} with $\b\in(0,\b_0)$, then there exists a standard double-bubble $\E_0$ with $\vol(\E_0)=(m_1,m_2)$ and a $C^{1,1}$-diffeomorphism $f_\b$ between $\pa\E_0$ and $\pa\E_\b$ such that
  \[
  \|f_\b-\Id\|_{C^0(\pa\E_0)}^3+\|\nabla f_\b-\Id\|_{C^0(\pa\E_0)}^6\le C_0\,\beta\,.
  \]
\end{theorem}

We now comment on the related literature on quantitative isoperimetric inequalities, and on the strategy of proof of our main results. After the pioneering contributions by Bernstein \cite{bernstein} and Bonnesen \cite{bonnesen}, the analysis of global stability problems has received a renewed attention in recent years, with the proof of the sharp stability inequality for the Euclidean isoperimetric problem \cite{Fuglede,Fuglede93,hallhaymanweitsman,hall,fuscomaggipratelli,CicaleseLeonardi,fuscogellipisante,fuscojulin}, the Wulff isoperimetric problem \cite{FigalliMaggiPratelliINVENTIONES}, the Gaussian isoperimetric problem \cite{cianchifuscomaggipratelliGAUSS,mossel,barchiesibrancolinijulin}, Plateau-type problems \cite{dephilippismaggi}, fractional isoperimetric problems \cite{fuscomillotmorini}, and isoperimetric problems in higher codimension \cite{bogelainduzaarfusco}. (This list is probably incomplete, and it does not mention contributions to stability problems for functional inequalities.)

Among the various methods developed to deal with global stability problems in the above mentioned papers, the {\it selection principle} method from \cite{CicaleseLeonardi} has proven to be the more widely applicable. At the heart of this approach lies the use of regularity theory to obtain what we call {\it improved convergence theorems}. Referring to the introduction of \cite{CiLeMaIC1} for a more detailed account on this kind of results, we just notice here that by exploiting the main result from \cite{CiLeMaIC1} in combination with a selection principle we can reduce the proof of \eqref{global stability double-bubble} to the case when $\pa\E=f(\pa\E_0)$ for a $C^{1,1}$-diffeomorphism $f$ between $\pa\E_0$ and $\pa\E$ such that $\|f-\Id\|_{C^1(\pa\E_0)}$ is as small as needed. In the case of the standard isoperimetric problem, following Fuglede \cite{Fuglede,Fuglede93}, one can directly address this ``reduced'' stability problem by an expansion in spherical harmonics, which is elementary if $n=2$.

In the case of double-bubbles, even when $n=2$, the situation is much subtler, due to the presence of singularities and of the multiple-volumes constraint. We shall address this problem by combining Fourier series arguments in the spirit of Fuglede with the solution of certain one-dimensional variational problems, to proceed through  a case by case analysis. Different cases will correspond to different behaviors of the perturbed interfaces, based for example on the relative size between their $L^2$-mean deviation and their $L^2$-distance from the corresponding interfaces of the reference standard double-bubble. The resulting argument, although based on rather elementary mathematical tools, sheds light on the non-trivial interactions between the three interfaces, on which the global stability of standard double-bubbles ultimately depends. As an entirely analogous structure underlies the stability problem for standard double-bubbles in higher dimensions, we expect the methods of this paper to be useful also in that case.

We notice that, at present, there is only another instance of isoperimetric problem with multiple volume constraints whose minimizers are explicitly known. This is the case of the planar triple bubble problem, addressed by Wichiramala in \cite{wichi}. It is reasonable to expect that by further exploiting the arguments developed in this paper, and again in combination with the improved convergence theorem from \cite{CiLeMaIC1}, one should be able to obtain results like Theorem \ref{thm main 2} and Theorem \ref{thm main perturbed} in the case of planar triple bubbles too.

The paper is organized as follows. In section \ref{section reduction to small diffeo} we reduce the proof of Theorem \ref{thm main 2} to the case of small diffeomorphic images of $\E_0$. In section \ref{section final} we introduce the notion of $(\e,\s)$-perturbation of a standard double-bubble, and prove Theorem \ref{thm main 2} and Theorem \ref{thm main perturbed} assuming Theorem \ref{thm main 2} on $(\e,\s)$-perturbations. Finally, in section \ref{section small perturbations}, we address the proof of Theorem \ref{thm main 2} on $(\e,\s)$-perturbations.

\subsection*{Acknowledgements} GPL is supported by the GNAMPA-INdAM project {\it Problemi di regolarit\`a e teoria geometrica della misura in spazi metrici} and by the PRIN 2010 M.I.U.R. project {\it Calcolo delle Variazioni}. FM is supported by NSF-DMS Grant 1265910 and  NSF-DMS FRG Grant 1361122.


\section{Reduction to small perturbations}\label{section reduction to small diffeo}

\subsection{Sets of finite perimeter, clusters, and improved convergence}\label{section sofp} We describe bubble clusters in the framework of the theory of sets of finite perimeter. Referring to \cite{maggiBOOK} for more details, given a set $E$ of locally finite perimeter in $\R^n$, we denote by $\mu_E=\nu_E\,\H^{n-1}\llcorner\pa^*E$ its Gauss--Green measure, where $\nu_E$ and $\pa^*E$ are the measure-theoretic outer unit normal and the reduced boundary of $E$, respectively. In this way the perimeter of $E$ relative to the Borel set $F$ is $P(E;F)=|\mu_E|(F)=\H^{n-1}(F\cap\pa^*E)$, and we set $P(E)=P(E;\R^n)$. We work under the normalization by a Lebesgue negligible set which ensures that
\[
\ov{\pa^*E}=\spt\,\mu_E=\big\{x\in A:0<|E\cap B_{x,r}|<\om_n\,r^n\quad\forall r>0\big\}=\pa E\,,
\]
Given a $N$-cluster $\E$ in $\R^n$, we set
\[
\pa^*\E=\bigcup_{h=1}^N\pa^*\E(h)\,,
\qquad\pa\E=\bigcup_{h=1}^N\pa\E(h)\,,\qquad\S(\E)=\pa\E\setminus\pa^*\E\,,
\]
so that $\ov{\pa^*\E}=\pa\E$. We set $\d(\E,\F)=(1/2)\,\sum_{h=0}^N|\E(h)\Delta\F(h)|$ for the $L^1$-distance between the $N$-clusters $\E$ and $\F$, and say that $\E$ is a $(\Lambda,r_0)$-minimizing cluster in $\R^n$ if
\begin{equation}
  \label{lambdar0 minimizer}
  P(\E)\le P(\F)+\Lambda\,\d(\E,\F)\,,
\end{equation}
whenever $\E(h)\Delta\F(h)\cc B_{x,r_0}$ for some $x\in\R^n$ and every $h=1,...,N$. Referring to \cite[Section 4]{CiLeMaIC1} for an account on the regularity properties of $(\Lambda,r_0)$-minimizing clusters in $\R^n$ for $n$ arbitrary, here we just need to recall what happens when $n=2$. Let us say that $\E$ is a $C^{k,\a}$-cluster in $\R^2$ ($k\in\N$, $\a\in(0,1]$) if there exist a locally finite family $\{\g_i\}_{i\in I}$ of closed $C^{k,\a}$-curves with boundary in $\R^2$ and a locally finite family of points $\{p_j\}_{j\in J}$ such that
\[
\pa\E=\bigcup_{i\in I}\g_i\,,\qquad\pa^*\E=\bigcup_{i\in I}\INT(\g_i)\,,\qquad \S(\E)=\bigcup_{i\in I}\bd(\g_i)=\bigcup_{j\in J}\{p_j\}\,,
\]
where $\INT(\g)$ and $\bd(\g)$ denote the interior and the boundary points of the curve $\g$. If $\E$ is a $(\Lambda,r_0)$-minimizing cluster in $\R^2$ then $\E$ is a $C^{1,1}$-cluster in $\R^2$: moreover, each $\g_i$ to have distributional curvature bounded by $\Lambda$, and each $p_j$ to be a boundary point of exactly three curves from $\{\g_i\}_{i\in I}$, which form three 120 degrees angles at $p_j$. For a proof of all these facts we refer, for example, to \cite[Theorem 5.2]{CiLeMaIC1}.

Given a $C^{1,1}$-cluster $\E$ in $\R^2$ and a map $f:\pa\E\to\R^2$ one says that $f\in C^{1,1}(\pa\E;\R^2)$ if $f$ is continuous on $\pa\E$ and
\[
\|f\|_{C^{1,1}(\pa\E)}:=\sup_{i\in I}\|f\|_{C^{1,1}(\g_i)}<\infty\,;
\]
moreover, given $C^{1,1}$-clusters $\E$ and $\F$, one says that $f$ is a $C^{1,1}$-diffeomorphism between $\pa\E$ and $\pa\F$ if $f$ is an homeomorphism between $\pa\E$ and $\pa\F$ with $f\in C^{1,1}(\pa\E;\R^2)$, $f^{-1}\in C^{1,1}(\pa\F;\R^2)$ and $f(\S(\E))=\S(\F)$. Finally, given a map $f:\pa\E\to\R^2$, and denoted by $\nu:\pa^*\E\to S^1$ a vector field with $\nu(x)\in\{\nu_{\E(h)}(x),\nu_{\E(k)}(x)\}$ for every $x\in\pa^*\E(h)\cap \pa^{*}\E(k)$, we define the tangential component $\ttau_\E:\pa^*\E\to\R^2$ of $f$ with respect to $\E$ by setting
\[
\ttau_\E f(x)=f(x)-(f(x)\cdot\nu(x))\nu(x)\qquad x\in\pa^*\E\,.
\]
(Note that the continuity of $\nu$ is not essential here, as $\ttau_\E f$ depends quadratically from $\nu$.) The following result is \cite[Theorem 1.5]{CiLeMaIC1}.

\begin{theorem}\label{thm IC1}
  Given $\Lambda\ge0$, $r_0>0$ and a bounded $C^{2,1}$-cluster $\E_0$ in $\R^2$, there exist positive constants $\mu_0$ and $C_0$ (depending on $\Lambda$ and $\E$) with the following property.

  If $\{\E_k\}_{k\in\N}$ is a sequence of $(\Lambda,r_0)$-minimizing clusters in $\R^2$ such that $\d(\E_k,\E_0)\to 0$ as $k\to\infty$, then  for every $\mu<\mu_0$ there exist $k(\mu)\in\N$ and a sequence of maps $\{f_k\}_{k\ge k(\mu)}$ such that each $f_k$ is a $C^{1,1}$-diffeomorphism between $\pa\E_0$ and $\pa\E_k$ with
  \begin{eqnarray}\label{hit 1}
  \|f_k\|_{C^{1,1}(\pa\E_0)}&\le&C_0\,,
  \\\label{hit 2}
  \lim_{k\to\infty}\|f_k-\Id\|_{C^1(\pa\E_0)}&=&0\,,
  \\\label{hit 3}
    \ttau_{\E_0}(f_k-\Id)&=&0\,,\qquad\mbox{on $\pa\E_0\setminus I_\mu(\S(\E_0))$}\,,
  \\\label{hit 4}
  \|\ttau_{\E_0}(f_k-\Id)\|_{C^1(\pa^*\E_0)}&\le&\frac{C_0}\mu\,\|f_k-\Id\|_{C^0(\S(\E_0))}\,.
  \end{eqnarray}
\end{theorem}

\subsection{A selection principle}\label{section select principle} Let now $\E_0$ denote a reference standard double-bubble in $\R^2$ with $\vol(\E_0)=(m_1,m_2)$, and for every planar double-bubble $\E$ set
\begin{gather*}
\de(\E)=P(\E)-P(\E_0)\,,
\\
\a(\E)=\inf\big\{\d(\E,f(\E_0)): \mbox{$f:\R^2\to\R^2$ is an isometry}\big\}\,,
\end{gather*}
and
\begin{equation}
  \label{kappa ottimale}
  \kappa(\E_0)=\inf\big\{\liminf_{k\to\infty}\frac{\de(\E_k)}{\a(\E_k)^2}:\vol(\E_k)=(m_1,m_2)\,,\a(\E_k)>0\,,\lim_{k\to\infty}\d(\E_k,\E_0)=0\big\}\,.
\end{equation}
Notice that, by pushing the interfaces of $\E_0$ as depicted in Figure \ref{fig pushing},
\begin{figure}
  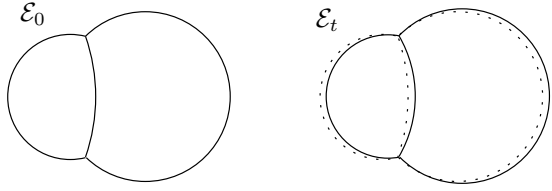\caption{\small{The deformations $\E_t$ of $\E_0$ used to prove that $\k(\E_0)<\infty$ is depicted on the right.}}\label{fig pushing}
\end{figure}
one defines a one-parameter family of double-bubbles $\{\E_t\}_{0<t<1}$ such that
\[
\vol(\E_t)=\vol(\E_0)\,,\qquad
P(\E_t)\le P(\E_0)+C\,t^2\,,\qquad \d(\E_t,\E_0)\ge C\,t\,,\qquad \forall t\in(0,1)\,;
\]
moreover, by exploiting the symmetry of $\E_t$ (see \cite[Lemma 5.2]{maggibams} for the kind of argument used here) one has
\[
\d(\E_t,\E_0)\le C\,\a(\E_t)\,,\qquad \forall t\in(0,1)\,,
\]
so that $\k(\E_0)<\infty$. This last fact shows, in particular, the sharpness of the decay rate in \eqref{global stability double-bubble} claimed in Remark \ref{remark sharp}. Now, Theorem \ref{thm main 2} is equivalent to $\k(\E_0)>0$, and Theorem \ref{thm reduction} below allows one to reduce the proof of Theorem \ref{thm main 2} to the case when $\pa\E$ is a $C^{1,1}$-diffeomorphic image of $\pa\E_0$ (in the sense of Theorem \ref{thm IC1}) by a map $f$ that is arbitrarily $C^1$-close to the identity.

\begin{theorem}\label{thm reduction} There exist positive constants $C_0$ and $\mu_0$ (depending on $m_1$ and $m_2$ only) and a sequence of planar double-bubbles $\{\E_k\}_{k\in\N}$ with $\vol(\E_k)=(m_1,m_2)$, such that
 \begin{equation}
   \label{x tesi uno}
    \inf_{k\in\N}\a(\E_k)>0\,,\qquad \lim_{k\to\infty}\d(\E_k,\E_0)= 0\,,\qquad
 \lim_{k\to\infty}\frac{\de(\E_k)}{\a(\E_k)^2}=\k(\E_0)\,,
 \end{equation}
 and such that for every $\mu\in(0,\mu_0)$ there exist $k(\mu)\in\N$ and, for each $k\ge k(\mu)$, a $C^{1,1}$-diffeomorphism $f_k$ between $\pa\E_0$ and $\pa\E_k$, in such a way that \eqref{hit 1}--\eqref{hit 4} hold.
\end{theorem}

\begin{proof} By Theorem \ref{thm selection principle part one} in Appendix \ref{appendix soft and select} there exists a sequence $\{\E_k\}_{k\in\N}$ of $(\Lambda,r_0)$-minimizing $2$-clusters in $\R^2$ with $\vol(\E_k)=(m_1,m_2)$ satisfying \eqref{x tesi uno}. Since $\d(\E_k,\E_0)\to 0$, by applying Theorem \ref{thm IC1} we find diffeomorphisms $f_k$ between $\pa\E_0$ and $\pa\E_k$ with the required properties.
\end{proof}

\section{Proofs of the main theorems}\label{section final} Given $\e>0$ and $\s\in(-1,1)$, and denoted by $\nu_{\E_0}$ a normal vector field to $\pa^*\E_0$, one says that a planar double-bubble $\E$ is an $(\e,\s)$-perturbation of $\E_0$ if $\vol(\E)=\vol(\E_0)$ and there exist $g\in C^1(\pa\E_0;\R^2)$ with
\begin{equation}\label{cool}
  g=\Id\quad\mbox{on $\S(\E_0)$}\,,\qquad (g-\Id)\cdot\nu_{\E_0}=0\quad\mbox{on $\pa^*\E_0$}\,,\qquad \|g-\Id\|_{C^1(\pa\E_0)}<\e\,,
\end{equation}
and such that $\pa\E=(1+\s)\,g(\pa\E_0)$. In the next section, see Theorem \ref{thm Fuglede double-bubble}, we show the existence of positive constants $\e_1$ and $\s_1$ such that \eqref{global stability double-bubble} hold on every $(\e,\s)$-perturbation of $\E_0$ with $\e<\e_1$ and $|\s|<\s_1$. Based on Theorem \ref{thm reduction} and Theorem \ref{thm Fuglede double-bubble} one can prove Theorem \ref{thm main 2} as follows.

\begin{proof}[Proof of Theorem \ref{thm main 2}]
  By Theorem \ref{thm reduction} and Theorem \ref{thm Fuglede double-bubble} it suffices to show that if $\{\E_k\}_{k\in\N}$ is a sequence of $(\Lambda,r_0)$-minimizing clusters such that $\d(\E_k,\E_0)\to 0$, then for every $k$ large enough $\E_k$ is an $(\e_k,\s_k)$-perturbation of $\E_0$, where $\e_k,\s_k\to 0$ as $k\to\infty$. In other words, we want to prove that, up to isometries, $\pa\E_k$ is a $C^1$-small {\it normal} perturbation of the small rescaling $(1+\s_k)\pa\E_0$ of $\pa\E_0$.

  We already know $\pa\E_k$ to be a $C^1$-small perturbation of $\pa\E_0$ with a small tangential displacement. Indeed, if $C_0$ and $\mu_0$ are as in Theorem \ref{thm IC1}, then by Theorem \ref{thm reduction} and for every $\mu<\mu_0$ we find $\{f_k\}_{k\ge k(\mu)}$ (the dependence of $f_k$ from $\mu$ is tacitly understood) such that \eqref{hit 1}--\eqref{hit 4} hold. We now exploit the existence of the maps $f_k$ to show that \eqref{cool} holds with $\E=\E_k$ for some $\s=\s_k\to 0$, $\e=\e_k\to 0$ and $g=g_k$.

  Let us set $\S(\E_0)=\{p_1,p_2\}$ and let $\{\g_i\}_{i=1}^3$ be the circular arcs such that $\pa\E_0=\bigcup_{i=1}^3\g_i$ and $\bd(\g_i)=\{p_1,p_2\}$ for $i=1,2,3$. Up to a translation of $\E_0$ (and, correspondingly, of each $\E_k$) we may assume that $p_1+p_2=0$. Setting $p_j^k=f_k(p_j^k)$, we have $\S(\E_k)=\{p_1^k,p_2^k\}$ and $p_j^k\to p_j$ by \eqref{hit 2}, so that, up to moving each $\E_k$ by an isometry (with the corresponding sequence of isometries which converges to the identity map) we entail
  \begin{equation}
    \label{hit 0}
    p_j^k=(1+\s_k)\,p_j\,,\qquad \lim_{k\to\infty}\s_k=0\,.
  \end{equation}
  If we set $\g_i^k=(1+\s_k)^{-1}\,f_k(\g_i)$, then
  \[
  (1+\s_k)^{-1}\pa\E_k=\bigcup_{i=1}^3\g_i^k\,,\qquad \bd(\g_i^k)=\{p_1,p_2\}\,.
  \]
  Thanks to \eqref{hit 1}--\eqref{hit 4}, by $\g_i^k=(1+\s_k)^{-1}\,f_k(\g_i)$, and since $\s_k\to 0$, one has:

  \medskip

  \noindent (i) if $\tau_\g:\bd(\g)\to S^1$ is the outer unit tangent vector to a curve $\g$ at its boundary points, then
  \[
  \lim_{k\to\infty}\hd(\g_i^k,\g_i)+\max_{j=1,2}|\tau_{\g_i^k}(p_j)-\tau_{\g_i}(p_j)|
  =0\,;
  \]
  moreover, by exploiting the fact that  $f_k$ parameterizes $\g_i^k$ over $\g_i$, one constructs unit normal vector fields $\nu_i^k\in C^{0,1}(\g_i^k;S^1)$ to $\g_i^k$ such that
  \[
  |\nu_i^k(x)\cdot(y-x)|\le L\,|x-y|^2\,,\qquad |\nu_i^k(x)-\nu_i^k(y)|\le L\,|x-y|\,,\qquad\forall x,y\in \g_i^k\,,
  \]
  where $L$ is independent from $k$;

  \medskip

  \noindent (ii) if we set $[\g_i]_t=\{x\in\g_i:\dist(x,\bd(\g_i))>t\}$, $t>0$, and $\psi_k=(1+\s_k)^{-1}\,(f_k-\Id)\cdot\nu_i$, then $\psi_k\in C^{1,1}([\g_i]_\mu)$ for every $i=1,2,3$ with
  \begin{gather*}
    \sup_{k\ge k(\mu)}\|\psi_k\|_{C^{1,1}([\g_i]_\mu)}\le C_0\,,\qquad \lim_{k\to\infty}\|\psi_k\|_{C^1([\g_i]_\mu)}=0\,,
\\    [\g_i^k]_{2\mu}\subset(\Id+\psi_k\nu_i)([\g_i]_\mu)\subset\g_i^k\,,
  \end{gather*}
  where $\nu_i\in C^{0,1}(\g_i;S^1)$ is a fixed outer unit normal to $\g_i$.

  \medskip

  \noindent Thanks to (i) and (ii) we can apply \cite[Theorem 3.5]{CiLeMaIC1} to construct a $C^{1,1}$-normal diffeomorphism $\hat g_i^k$ between $\g_i$ and $\g_i^k$ such that $\hat g_i^k\to\Id$ in $C^1(\g_i)$. Note that, in fact, $\hat g_i^k$ is a normal diffeomorphism as $\bd(\g_i)=\bd(\g_i^k)$, cf. with \cite[Equation (3.85)]{CiLeMaIC1}. Setting $g_k=\hat g_i^k$ on $\g_i$, we thus define a normal $C^{1,1}$-diffeomorphism between $\pa\E_0$ and $(1+\s_k)^{-1}\pa\E_k$ with $\e_k=\|g_k-\Id\|_{C^1(\pa\E_0)}\to 0$.
\end{proof}

\begin{proof}
  [Proof of Theorem \ref{thm main perturbed}] We directly focus on the case $m_2>m_1$, the case $m_2=m_1$ being analogous. Let us pick an arbitrary sequence $\beta_k\to 0^+$, and let $\E_k$ be minimizers in \eqref{isoperimetric problem with two chambers plus potential} with $\b=\b_k$. By arguing as in \cite[Proof of Theorem 1.10]{CiLeMaIC1} we prove the existence of $\Lambda\ge0$ and $r_0>0$ such that $\{\E_k\}_{k\in\N}$ is a sequence of $(\Lambda,r_0)$-minimizers such that, up to isometries, $\d(\E_k,\E_0)\to 0$. By the argument used to prove Theorem \ref{thm main 2}, we see that $\E_k$ is an $(\e_k,\s_k)$-perturbation of $\E_0$ with $\e_k,\s_k\to 0$. As a first consequence, we note that if $R>0$ is such that $\E_0(h)\subset B_R$ for $h=1,2$, then for $k$ large enough $\E_k(h)\subset B_{2R}$ for $h=1,2$, and thus by minimality of $\E_k$,
  \[
  P(\E_k)-P(\E_0)\le C\,\beta_k\,\|J\|_{C^0(B_{2R})}\,\sum_{h=1}^2\,|\E_k(h)\Delta\E_0(h)|\le C\,\beta_k\,.
  \]
  At the same time, if with the same notation of the previous proof we denote by $\{\g_i\}_{i=0}^2$ the circular arcs composing $\pa\E_0$, then there exist $u_{k,i}\in C^{1,1}_0(\g_i)$ such that
  \begin{equation}
    \label{ictp}
          \lim_{k\to\infty}\|u_{k,i}\|_{C^1(\g_i)}=0\,,\qquad \sup_{k\in\N}\|u_{k,i}''\|_{L^\infty(\g_i)}\le \Lambda\,,\qquad\forall i=0,1,2\,,
  \end{equation}
  and such that, by setting
  \[
  \bar g_k(x)=(1+\s_k)\,\big(x+u_{k,i}(x)\,\nu_i(x)\big)\,,\qquad x\in\g_i\,,
  \]
  one defines a $C^{1,1}$-diffeomorphism $\bar g_k$ between $\pa\E_0$ and $\pa\E_k$ with
  \begin{equation}
  \label{basta}
  \|\bar g_k-\Id\|_{C^j(\pa\E_0)}\le C\,\Big(|\s_k|+\sum_{i=0}^2\|u_{k,i}\|_{C^j(\g_i)}\Big)\,,\qquad j=1,2\,,
  \end{equation}
  Since $\e_k,\s_k\to 0$, for $k$ large enough we can use Theorem \ref{thm Fuglede double-bubble} to deduce that
  \begin{eqnarray}\label{ohoh}
  P(\E_k)-P(\E_0)\ge \k\,\Big(\s_k^2+\sum_{i=0}^2\int_{\g_i}\,u_{k,i}^2\Big)\,,
  \end{eqnarray}
  and then apply Lemma \ref{lemma:interpol} below to get
  \[
  \|\bar g_k-\Id\|_{C^0(\pa\E_0)}^3+\|\nabla\bar g_k-\Id\|_{C^0(\pa\E_0)}^6\le C\,\b_k\,.
  \]
  By the arbitrariness of $\beta_k$ we conclude the proof of the theorem.
  \end{proof}
  
  \begin{lemma}\label{lemma:interpol}
  If $v\in C^{1,1}([a,b])$ with $v(a)=v(b)=0$, then
  \begin{equation}\label{eq:interpol}
  \begin{split}
  C\,\|v\|_{L^{1}(a,b)}^{2/3}\,\|v''\|_{L^\infty(a,b)}^{1/3} \ge \|v\|_{C^0([a,b])}\,,
  \\
  C\,\|v\|_{L^{1}(a,b)}^{1/3}\,\|v''\|_{L^\infty(a,b)}^{2/3} \ge \|v'\|_{C^0([a,b])}\,.
  \end{split}
  \end{equation}
  \end{lemma}
  
\begin{proof} The argument is elementary and it is included just for the sake of clarity. Without loss of generality, let $x_{0}\in (a,b)$ be such that $\|v\|_{C^{0}([a,b])}=|v(x_{0})|=v(x_0)>0$. Since $v(b)=0$, there exists $\bar x\in(x_0,b]$ such that $v>0$ on $(x_0,\bar x)$ and $v(\bar x)=0$. By $v'(x_0)=0$ we find
\[
|v(x)|=v(x)\ge v(x_{0})-\frac{\|v''\|_{L^\infty(a,b)}}{2}(x-x_{0})^{2}\,,\qquad\forall x\in(x_0,\bar x)\,.
\]
The right-hand side of this inequality is positive for $x\in(x_0,x_0+r)$ where
\[
r=\Big(\frac{2\|v\|_{C^0([a,b])}}{\|v''\|_{L^\infty(a,b)}}\Big)^{1/2}\,,
\]
hence $(x_0,x_0+r)\subset(x_0,\bar x)$, and thus
\[
\|v\|_{L^1(a,b)}\ge\int_{(x_0,x_0+r)}\Big(v(x_{0})-\frac{\|v''\|_{L^\infty(a,b)}}{2}(x-x_{0})^{2}\Big)\,dx
=\frac{2\sqrt2}3\,\frac{\|v\|_{C^0([a,b])}^{3/2}}{\|v''\|_{L^\infty(a,b)}^{1/2}}\,.
\]
which is the first estimate in \eqref{eq:interpol}. Now we take $x_{1}\in [a,b]$ such that $|v'(x_{1})| = \|v'\|_{C^{0}([a,b])}$. Without loss of generality we can assume that $|v'(x_1)|=v'(x_1)>0$ and that $v(x_1)\ge0$. (Indeed, this can be achieved by possibly replacing $v$ with $-v$ and then by reflecting $v$ with respect to the mid-point of $[a,b]$. Notice that this operation may in principle change the sign of $v(x_0)$, but this will not affect our argument as we shall not need to refer to $v(x_0)$ anymore.) Since $v(b)=0$, there exists $x_{2}\in(x_{1},b)$ such that $v'=|v'|>0$ on $(x_1,x_2)$ and $v'(x_2)=0$, and thus, by $v(x_1)\ge0$, $|v|=v$ on $(x_1,x_2)$. In particular,
\begin{eqnarray*}
|v(x)|&=&v(x)\ge v(x_1)+v'(x_1)(x-x_1)-\frac{\|v''\|_{L^\infty(a,b)}}2\,|x-x_1|^2
\\
&\ge& v'(x_1)(x-x_1)-\frac{\|v''\|_{L^\infty(a,b)}}2\,|x-x_1|^2\,,\qquad\forall x\in(x_1,x_2)\,,
\end{eqnarray*}
where the right-hand side of this inequality is non-negative for $x\in(x_1,x_1+s)$, where
\[
s=\frac{2\|v'\|_{C^0([a,b])}}{\|v''\|_{L^\infty(a,b)}}\,.
\]
In particular $(x_1,x_1+s)\subset(x_1,x_2)$, and thus
\[
\|v\|_{L^1(a,b)}\ge \int_{(x_1,x_1+s)}\Big(v'(x_1)(x-x_1)-\frac{\|v''\|_{L^\infty(a,b)}}2\,|x-x_1|^2\Big)\,dx
=\frac{2}3\,\frac{\|v'\|_{C^0([a,b])}^{3}}{\|v''\|_{L^\infty(a,b)}^2}\,.
\]
\end{proof}

\section{Stability on $(\e,\s)$-perturbations}\label{section small perturbations} We now turn to the proof of Theorem \ref{thm main 2} on $(\e,\s)$-perturbations of $\E_0$, see Theorem \ref{thm Fuglede double-bubble} below. We begin by introducing some specific notation for spherical caps and sectors, and for their normal perturbation by a given function. Let $B=\{x\in\R^2:|x|<1\}$. Given $\theta\in(0,\pi)$, we define a circular arc $A(\theta)\subset \partial B$ and a circular sector $S(\theta)\subset B$ by setting
\begin{eqnarray*}
A(\theta)=\big\{x\in \R^2: |x|=1\,, x_1>\cos\theta\big\}\,,
\qquad S(\theta)=\big\{t\,x: x\in A(\theta)\,,0<t<1\big\}\,,
\end{eqnarray*}
while, given $u\in W^{1,2}_0(A(\theta))$ we denote by $A(\theta,u)\subset \R^2$ and $S(\theta,u)\subset \R^2$ the perturbed circular arc and perturbed circular sector defined as
\begin{eqnarray*}
A(\theta,u)=\big\{(1+u(x))\,x:x\in A(\theta)\big\}\,,
\qquad S(\theta,u)=\big\{t\,(1+u(x))\,x:x\in A(\theta)\,, 0<t<1\big\}\,;
\end{eqnarray*}
see Figure \ref{fig: calotta}.
\begin{figure}
  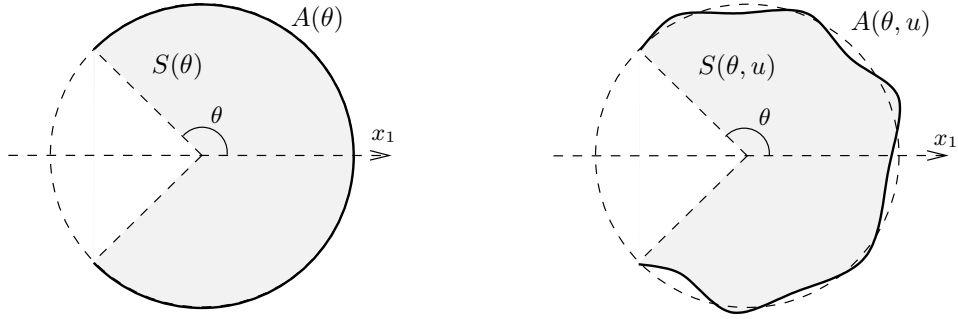\caption{{\small The circular arc $A(\theta)$, the circular sector $S(\theta)$, and their perturbations defined by $u\in W^{1,2}_0(A(\theta))$.}}\label{fig: calotta}
\end{figure}
(Notice that $A(\theta,0)=A(\theta)$ and $S(\theta,0)=S(\theta)$.) In the analysis of the case $m_1=m_2$, where the interface between the chambers is a segment, it is convenient to introduce as a reference domain the vertical open segment $H$ and its perturbations $H(u)$ defined as
\begin{eqnarray}
H=\big\{x\in\R^2:\, |x_2|<\frac{\sqrt{3}}2,\, x_1=0\big\}\,,\qquad H(u)=\big\{x+u(x)\,e_1: x\in \,H\big\}\,,
\end{eqnarray}
in correspondence of $u\in W^{1,2}_0(H)$. We occasionally identify $A(\theta)$ with the interval $(-\theta,\theta)$ and $H$ with the interval $(-\sqrt{3}/2,\sqrt{3}/2)$; correspondingly, we identify $W_0^{1,2}(A(\theta))$ with $W_0^{1,2}(-\theta,\theta)$ and $W^{1,2}_0(H)$ with $W^{1,2}_0(-\sqrt{3}/2,\sqrt{3}/2)$.

\begin{lemma}\label{lemma: sviluppi calotte}
If $u\in C^1_0(-\theta,\theta)$, then
  \begin{eqnarray}\label{g theta u}
    |S(\theta,u)|-|S(\theta)|&=&\int_{-\theta}^{\theta} u+\frac{u^2}2 \,,
    \\\label{kappa theta u}
    \H^{1}(A(\theta,u))-\H^{1}(A(\theta))&=&\int_{-\theta}^{\theta} u+\frac{(u')^2}2
   \, +\|u\|_{C^1(-\theta,\theta)}\,\O(\|u\|_{W^{1,2}(-\theta,\theta)}^2)\,.\hspace{0.5cm}
  \end{eqnarray}
Moreover, if $|u|\le 1$, then
\begin{equation}
  \label{giovani inesperti}
  |S(\theta,u)\Delta S(\theta)|\le \frac32\,\int_{-\theta}^{\theta}|u|\,.
\end{equation}
 \end{lemma}

\begin{proof}
 Identity \eqref{g theta u} follows from $|S(\theta,u)|=(1/2)\int_{-\theta}^{\theta}(1+u)^2$, which also implies \eqref{giovani inesperti} since, if $|u|\le 1$, then
  \[
  |S(\theta,u)\Delta S(\theta)|=\int_{-\theta}^\theta\Big|\frac{(1+u)^2-1}2\Big|\le \frac32\,\int_{-\theta}^{\theta}|u|\,.
  \]
  Concerning \eqref{kappa theta u}, we notice that $A(\theta,u)=T(A(\theta))$ where we have set $T:A(\theta)\to A(\theta, u)$, $T(x)=(1+u(x))x$, $x\in A(\theta)$. The Jacobian of $T$ on $A(\theta)$ is $J T=\sqrt{(1+u)^{2}+|u'|^2}$, and thus \eqref{kappa theta u} follows from $\sqrt{1+t}=1+(t/2)-(t^2/8)+\O(t^3)$.
\end{proof}

Next, given $m_2\ge m_1>0$, we fix a reference standard double-bubble $\E_0$ with $\vol(\E_0)=(m_1,m_2)$ by requiring that the two point singularities of $\E_0$ belong to the $x_2$-axis, and that their middle-point lies at the origin (indeed, these geometric requirements uniquely identify $\E_0$). In the case that $m_2>m_1$, there exist $L_k:\R^2\to\R^2$ isometries, $r_k>0$, and $\theta_k\in(0,\pi)$ such that
\begin{eqnarray}\label{e0 1}
  \pa \E_0(1)\cap\pa \E_0(2)&=& L_0\,r_0\,A(\theta_0)\,,
  \\\label{e0 2}
  \pa \E_0(1)\setminus\pa \E_0(2)&=& L_1\,r_1\,A(\theta_1)\,,
  \\\label{e0 3}
  \pa \E_0(2)\setminus\pa \E_0(1)&=& L_2\,r_2\,A(\theta_2)\,.
\end{eqnarray}
With reference Figure \ref{fig: doppiabolla}, we thus have
\begin{eqnarray*}
  r_0=|S-P_0|\,,&&\quad \theta_0=(P_1P_0S)\,,
  \\
  r_1=|S-P_1|\,,&&\quad \theta_1=(P_0P_1S)\,,
  \\
  r_2=|S-P_2|\,,&&\quad \theta_2=\pi-(P_1P_2S)\,,
\end{eqnarray*}
and it holds
\begin{equation}
  \label{theta012+}
  r_0\sin\theta_0=r_1\sin\theta_1\,,\quad\quad r_0\sin\theta_0=r_2\sin\theta_2\,.
\end{equation}
By Plateau's laws (vanishing of first variation), the three circular arcs meet at 120 degrees angles,
\begin{equation}
  \label{theta012}
  \theta_1+\theta_0=\frac{2\pi}3\,,\quad\quad\theta_2-\theta_0=\frac{2\pi}3\,,
\end{equation}
and, correspondingly, the following inequalities hold true
\begin{equation}
  \label{theta intervalli}
  0<\theta_0<\frac{\pi}3\,,\quad\quad \frac{\pi}3<\theta_1<\frac{2\pi}3\,,\quad\quad \frac{2\pi}3<\theta_2<\pi\,.
\end{equation}
Vanishing of first variation also implies the following ``law of pressures'',
\begin{eqnarray}
  \label{vincolo pressioni}
  \frac1{r_1}=\frac1{r_2}+\frac1{r_0}\,.
\end{eqnarray}
Identities \eqref{theta012+} and \eqref{theta012} provide
\begin{figure}
  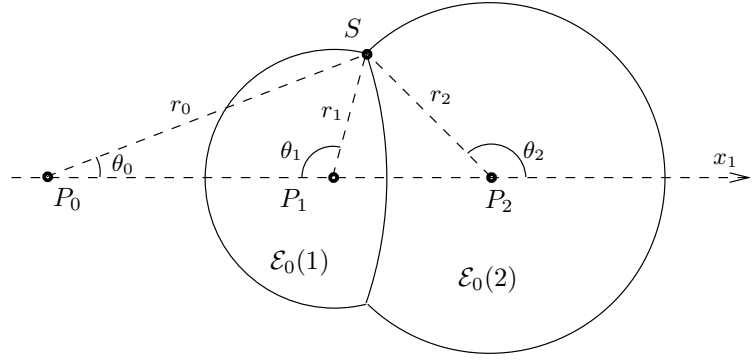\caption{{\small The reference standard double-bubble $\E_0$}.}\label{fig: doppiabolla}
\end{figure}
four constraints on the six parameters $r_k$ and $\theta_k$, $k=0,1,2$. Up to a scaling, which leaves the ratio $m_2/m_1$ invariant, we may add to \eqref{theta012+} and \eqref{theta012} a fifth constraint by requiring that
\[
r_2=1\,.
\]
This choice allows to express the remaining five parameters as functions of $r_1\in(0,1)$:
\begin{eqnarray}\label{r2}
 r_0&=&\frac{r_1}{1-r_1}\,,
 \\\label{r yeah}
 \theta_0&=&\arctan\Big(\frac{1-r_1}{1+r_1}\sqrt{3}\Big)\,,
 \\\label{r yeah 1}
 \theta_1&=&\frac{2\pi}3-\theta_0\,,
 \\\label{r yeah 2}
 \theta_2&=&\frac{2\pi}3+\theta_0.
\end{eqnarray}
Finally, in the case $m_1=m_2$, we set $m=m_1=m_2$, $r=r_1=r_2$, we have
\[
\theta_1=\theta_2=\frac{2\pi}{3}\,, \qquad \theta_0=0 \,,\qquad r_0=+\infty\,,
\]
and describe the interfaces of the reference standard double-bubble $\E_0$ as
\begin{eqnarray}\label{e0 4}
  \pa \E_0(1)\cap\pa \E_0(2)&=&L_0\,r\,H,
  \\\label{e0 5}
  \pa \E_0(1)\setminus\pa \E_0(2)&=& L_1\,r\,A\left(\frac{2\pi}{3}\right),
  \\\label{e0 6}
  \pa \E_0(2)\setminus\pa \E_0(1)&=& L_2\,r\,A\left(\frac{2\pi}{3}\right),
\end{eqnarray}
for some isometries $L_k:\R^2\to\R^2$, $k=0,1,2$; see Figure \ref{fig sym}.
\begin{figure}
  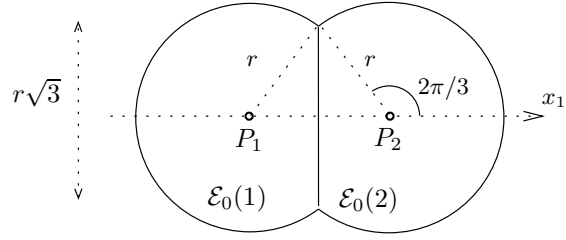\caption{\small{The reference standard double-bubble $\E_0$ with $m_1=m_2$.}}\label{fig sym}
\end{figure}
Notice that \eqref{e0 5} and \eqref{e0 6} are obtained from \eqref{e0 2} and \eqref{e0 3} by setting $\theta_1=\theta_2=(2/3)\pi$, while \eqref{e0 4} is not directly related to \eqref{e0 1}. Finally, we show the following useful formula for $P(\E_0)$ in terms of $m_1$, $m_2$, $r_1$, and $r_2$.

\begin{lemma} If $\E_0$ is the standard double-bubble with $m_2> m_1$, then
\begin{eqnarray}
  \label{gippo rule}
  P(\E_0)&=&2\Big(\frac{m_1}{r_1}+\frac{m_2}{r_2}\Big)\,,
  \\\label{formulam1}
  m_1&=&\theta_1\,r_1^2+\theta_0\,r_0^2-\frac{\sqrt{3}}2\,r_0\,r_1\,,
  \\\label{formulam2}
  m_2&=&\theta_2\,r_2^2-\theta_0\,r_0^2+\frac{\sqrt{3}}2\,r_0\,r_2\,.
\end{eqnarray}
Moreover, \eqref{gippo rule} holds true also when $m_2=m_1=m$, and in that case, we have
\begin{equation}
  \label{formulam}
  m=\Big(\frac{2\pi}3+\frac{\sqrt{3}}{4}\Big)\,r^2\,.
\end{equation}
\end{lemma}

\begin{proof}
  We apply the divergence theorem on the chamber $\E_0(1)$ to the vector field $x-P_1$, and on the chamber $\E_0(2)$ to the vector field $x-P_2$, to find that
  \begin{eqnarray}\label{m1}
  2\,m_1&=&2\theta_1\,r_1^2+\int_{\pa\E_0(1)\cap\pa\E_0(2)}(x-P_1)\cdot\nu_{\E_0(1)}(x)\,d\H^1(x)\,,
  \\\label{m2}
  2\,m_2&=&2\theta_2\,r_2^2+\int_{\pa\E_0(1)\cap\pa\E_0(2)}(x-P_2)\cdot(-\nu_{\E_0(1)}(x))\,d\H^1(x)\,.
  \end{eqnarray}
  (Here, $\nu_{\E_0(1)}$ denotes the outer unit normal to $\E_0(1)$.) In the case $m_2>m_1$, we set the origin at $P_0$ (see Figure \ref{fig: doppiabolla}), and parameterize $\pa\E_0(1)\cap\pa\E_0(2)$ as $\{r_0\,e^{i\theta}:|\theta|<\theta_0\}$. In this way, see Figure \ref{fig angoli},
  \begin{figure}
    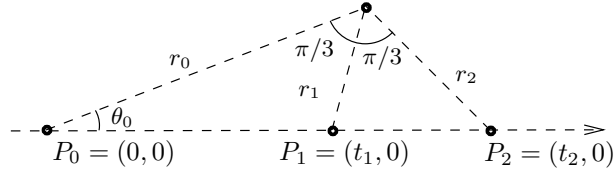\caption{\small{We have $t_1/\sin(\pi/3)=r_1/\sin\theta_0$ and $t_2/\sin(2\pi/3)=r_2/\sin\theta_0$.}}\label{fig angoli}
  \end{figure}
  we have $P_1=(t_1,0)$ and $P_2=(t_2,0)$, where
  \[
  \frac{t_1}{\sin(\pi/3)}=\frac{r_1}{\sin\theta_0}\,,\qquad\frac{t_2}{\sin(2\pi/3)}=\frac{r_2}{\sin\theta_0}\,,
  \]
  and, correspondingly
  \begin{eqnarray*}
    \int_{\pa\E_0(1)\cap\pa\E_0(2)}(x-P_1)\cdot\nu_{\E_0(1)}(x)\,d\H^1(x)&=&\int_{-\theta_0}^{\theta_0} (r_0\,e^{i\theta}-(t_1,0))\cdot e^{i\,\theta}\,r_0\,d\theta
    \\
    &=&2\theta_0\,r_0^2-2\sin\theta_0\,r_0\,t_1=2\theta_0\,r_0^2-\sqrt{3}\,r_0\,r_1\,,
    \\
    \int_{\pa\E_0(1)\cap\pa\E_0(2)}(P_2-x)\cdot\nu_{\E_0(1)}(x)\,d\H^1(x)&=&\int_{-\theta_0}^{\theta_0} ((t_2,0)-r_0\,e^{i\theta})\cdot e^{i\,\theta}\,r_0\,d\theta
    \\
    &=&-2\theta_0\,r_0^2+2\sin\theta_0\,r_0\,t_2=-2\theta_0\,r_0^2+\sqrt{3}\,r_0\,r_2\,.
  \end{eqnarray*}
  We plug these identities into \eqref{m1} and \eqref{m2} to find \eqref{formulam1} and \eqref{formulam2}; moreover, dividing \eqref{formulam1} and \eqref{formulam2} by $r_1$ and $r_2$ respectively, by adding up the resulting inequalities, and by \eqref{vincolo pressioni},
  \begin{eqnarray*}
    2\Big(\frac{m_1}{r_1}+\frac{m_2}{r_2}\Big)=2\theta_1\,r_1+2\theta_2\,r_2+2\theta_0\Big(\frac{r_0^2}{r_1}-\frac{r_0^2}{r_2}\Big)
    =2\theta_1\,r_1+2\theta_2\,r_2+2\theta_0\,r_0=P(\E_0)\,,
  \end{eqnarray*}
  that is \eqref{gippo rule}. In the case $m_2=m_1$, $\nu_{\E_0(1)}(x)=e_1$ and $(x-P_1)\cdot e_1=(P_2-x)\cdot e_1=\ell$ for every $x\in\pa\E_0(1)\cap\pa\E_0(2)$, where, by Pythagoras' theorem, $\ell=r/2$. Therefore, \eqref{m1} gives
  \[
  2\,m=2\,\frac{2\pi}3\,r^2+\ell\,\H^1(\pa\E_0(1)\cap\pa\E_0(2))=\frac{4\pi}3\,r^2+\frac{\sqrt{3}}{2}\,r^2=\frac{P(\E_0)}2\,r\,,
  \]
  and \eqref{gippo rule} holds true when $m_2=m_1$ too.
\end{proof}

We now describe the generic $(\e,\s)$-perturbation of $\E_0$ by means of the coordinates introduced above. Let $\E$ be a planar double-bubble with $\vol(\E)=\vol(\E_0)=(m_1,m_2)$. If $m_2>m_1$, then $\E$ is an $(\e, \s)$-perturbation of $\E_0$ if there exist functions $u_k\in C^1_0(A(\theta_k))$ with $\|u_k\|_{C^1}\le\e$ ($k=0,1,2$),
such that (compare with \eqref{e0 1}, \eqref{e0 2}, and \eqref{e0 3}),
\begin{eqnarray}\label{bbb1}
  \pa \E(1)\setminus\pa \E(2)&=&(1+\s)\,L_{1}\,r_1\,A(\theta_1,u_1)\,,
  \\\label{bbb2}
  \pa \E(2)\setminus\pa \E(1)&=&(1+\s)\,L_{2}\,r_2\,A(\theta_2,u_2)\,,
  \\\label{bbb3}
    \pa \E(1)\cap\pa \E(2)&=&(1+\s)\,L_{0}\,r_0\,A(\theta_0,u_0)\,.
\end{eqnarray}
If $m_2=m_1$, then $\E$ is an $(\e, \s)$-perturbation of $\E_0$ provided there exist functions $v_0\in C^1_0(H)$, and $u_k\in C^1_0(A(\theta_k))$, $\|v_0\|_{C^1}\le \e$ and $\|u_k\|_{C^1}\le\e$ ($k=1,2$), such that \eqref{bbb1} and \eqref{bbb2} hold true for $u_1$ and $u_2$, and, moreover (compare with \eqref{e0 4}), $\pa \E(1)\cap\pa \E(2)=(1+\s)\,L_{0}\,r\,H(v_0)$.

\begin{lemma}\label{lemma: deficit}
  If $\E$ is an $(\e,\s)$-perturbation of $\E_0$ and $m_2>m_1$, then
  \begin{eqnarray}
  \frac{P(\E)-P(\E_0)}{1+\s}= \sum_{k=0}^2\,r_k\int_{-\theta_k}^{\theta_k}\Big(\frac{(u'_k)^2}2-\frac{u_k^2}2\Big)+\frac{\s^2}{2}\,P(\E_0)
  \label{stima deficit 1}
  +\e\,\O(\|u\|_{W^{1,2}}^2)+\O(|\s|^3)\,.
  \end{eqnarray}
  If otherwise $m_2=m_1$ (and we set $r_1=r_2=r$), then we have
   \begin{eqnarray}\nonumber
  \frac{P(\E)-P(\E_0)}{1+\s}
  &=&r\,\int_{-\,\sqrt{3}/2}^{\,\sqrt{3}/2}\frac{(v'_0)^2}{2}+
  r\,\sum_{k=1}^2\,\int_{-2\pi/3}^{2\pi/3}\Big(\frac{(u'_k)^2}2-\frac{u_k^2}2\Big)+\frac{\s^2}{2}\,P(\E_0)
  \\ \label{stima deficit 1bis}
  &&+\e\,\O(\|u\|_{W^{1,2}}^2)+\O(|\s|^3)\,.
  \end{eqnarray}
  Here we have set
  \[
  \|u\|_{W^{1,2}}^2=
  \left\{
  \begin{split}
  &\sum_{k=0}^2\int_{\theta_k}^{\theta_k} u_k^2+(u'_k)^2\,,\qquad&\mbox{if $m_2>m_1$}\,,
  \\
  &\int_{-\,\sqrt{3}/2}^{\,\sqrt{3}/2} v_0^2+(v'_0)^2+\sum_{k=1}^2\int_{-2\pi/3}^{2\pi/3} u_k^2+(u'_k)^2\,,&\qquad\mbox{if $m_2=m_1$}\,,
  \end{split}
  \right .
    \]
\end{lemma}
\begin{proof}
We just give the details for the case $m_2>m_1$. By \eqref{kappa theta u}, \eqref{bbb1}, \eqref{bbb2} and \eqref{bbb3},
  \begin{eqnarray*}
    P(\E)-P((1+\s)\E_0)&=&(1+\s)\sum_{k=0}^2\,r_k\Big(\H^1(A(\theta_k,u_k))-\H^1(A(\theta_k))\Big)\,,
    \\
    &=&(1+\s)\sum_{k=0}^2\,r_k\int_{-\theta_k}^{\theta_k}\left(\frac{(u_k')^2}{2}+u_k\right)+\e\,\O(\|u\|_{W^{1,2}}^2)\,.
  \end{eqnarray*}
Therefore we may write
  \begin{eqnarray}\nonumber
    \frac{P(\E)-P(\E_0)}{1+\s}
    &=&\nonumber\sum_{k=0}^2\,r_k\int_{-\theta_k}^{\theta_k}\left(\frac{(u_k')^2}{2}+u_k\right)+(\s-\s^2)\,P(\E_0)+\e\,\O(\|u\|_{W^{1,2}}^2)+\O(|\s|^3)
    \\
    &=& \label{levico 111} \sum_{k=0}^2\,r_k\int_{-\theta_k}^{\theta_k}\left(\frac{(u_k')^2}{2}-\frac{u^2_k}{2}\right)
    +\sum_{k=0}^2\,r_k\int_{-\theta_k}^{\theta_k}\left(\frac{u_k^2}{2}+u_k\right)
    \\
    &&\nonumber+(\s-\s^2) P(\E_0)+\e\,\O(\|u\|_{W^{1,2}}^2)+\O(|\s|^3)\,.
  \end{eqnarray}
Again by \eqref{bbb1}, \eqref{bbb2} and \eqref{bbb3} we find that
   \begin{eqnarray}\label{vincolo volume E1 1}
    |\E(1)|-(1+\s)^2|\E_0(1)|&=&(1+\s)^2r_1^2\Big(|S(\theta_1,u_1)|-|S(\theta_1)|\Big)
    \\\nonumber
    &&+(1+\s)^2r_0^2\Big(|S(\theta_0,u_0)|-|S(\theta_0)|\Big)\,,
  \end{eqnarray}
  \begin{eqnarray}\label{vincolo volume E2 1}
    |\E(2)|-(1+\s)^2|\E_0(2)|&=&(1+\s)^2r_2^2\Big(|S(\theta_2,u_2)|-|S(\theta_2)|\Big)
    \\\nonumber
    &&-(1+\s)^2r_0^2\Big(|S(\theta_0,u_0)|-|S(\theta_0)|\Big)\,.
  \end{eqnarray}
 Since $\vol(\E)=\vol(\E_0)=(m_1,m_2)$, by \eqref{g theta u}, \eqref{vincolo volume E1 1} and \eqref{vincolo volume E2 1} we infer
  \begin{eqnarray}\label{gippo1}
  \Big(\frac1{(1+\s)^2}-1\Big)\,m_1&=&r_1^2\int_{-\theta_1}^{\theta_1}\left(u_1+\frac{u_1^2}2\right)+r_0^2\int_{-\theta_0}^{\theta_0}\left(u_0+\frac{u_0^2}2\right)\,,
  \\\label{gippo2}
  \Big(\frac1{(1+\s)^2}-1\Big)\,m_2&=&r_2^2\int_{-\theta_2}^{\theta_2}\left(u_2+\frac{u_2^2}2\right)-r_0^2\int_{-\theta_0}^{\theta_0}\left(u_0+\frac{u_0^2}2\right)\,.
  \end{eqnarray}
We now divide \eqref{gippo1} and \eqref{gippo2} by $r_1$ and $r_2$ respectively and sum the resulting identities to find that
  \begin{eqnarray*}
    \Big(\frac1{(1+\s)^2}-1\Big)\,\left(\frac{m_1}{r_1}+\frac{m_2}{r_2}\right)&=&r_1\int_{-\theta_1}^{\theta_1}\left(u_1+\frac{u_1^2}2\right)+r_2\int_{-\theta_2}^{\theta_2}\left(u_2+\frac{u_2^2}2\right)\\
    &+&\left(\frac1{r_1}-\frac1{r_2}\right)r_0^2\int_{-\theta_0}^{\theta_0}\left(u_0+\frac{u_0^2}2\right)\,.
  \end{eqnarray*}
Taking into account \eqref{vincolo pressioni} and \eqref{gippo rule} we conclude that
  \begin{eqnarray*}
    \Big(\frac1{(1+\s)^2}-1\Big)\,\frac{P(\E_0)}2=\sum_{k=0}^2r_k \int_{-\theta_k}^{\theta_k}\left(u_k+\frac{u_k^2}2\right)\,.
  \end{eqnarray*}
 Plugging this relation into \eqref{levico 111} we find
  \begin{eqnarray}\label{levico 2}
    \frac{P(\E)-P(\E_0)}{(1+\s)}
    &=&\sum_{k=0}^2\,r_k\int_{-\theta_k}^{\theta_k}\left(\frac{(u'_k)^2}{2}-\frac{u_k^2}2\right)
    \\\nonumber
    &&+\Big(\Big(\frac1{(1+\s)^2}-1\Big)\,+2(\s-\s^2)\Big)\frac{P(\E_0)}2+\e\,\O(\|u\|_{W^{1,2}}^2)+\O(|\s|^3)\,.
  \end{eqnarray}
We conclude the proof since $((1+\s)^{-2}-1)+2(\s-\s^2)=\s^2+\O(|\s|^3)$.
\end{proof}

We now provide an upper bound on the relative asymmetry of an $(\e,\s)$-perturbation of $\E_0$.

\begin{lemma}\label{lemma upper} There exists a constant $C$ (depending on $m_1/m_2$ only) with the following property. If $\E$ is an $(\e,\s)$-perturbation of $\E_0$ with $|\s|<1/2$, then, in case $m_2>m_1$,
\begin{equation}
  \label{alphaaa1}
  \a(\E)^2\le C\Big(m_2^2\s^2+\sum_{k=0}^2\,r_k^4\theta_k\,\int_{-\theta_k}^{\theta_k}\,u_k^2\Big)\,,
\end{equation}
while, in case $m_2=m_1=m$, setting $r_1=r_2=r$,
\[
\a(\E)^2\le C\,\Big(m^2\,\s^2+r^4\sum_{k=1}^2\,\int_{-2\pi/3}^{2\pi/3}\,u_k^2+r^4\int_{-\sqrt{3}/2}^{\sqrt{3}/2}\,v_0^2\Big)\,.
\]
\end{lemma}

\begin{proof} We just address the case $m_2>m_1$. Since
\[
|\E(1)\Delta (1+\s)\E_0(1)|=(1+\s)^2\,\sum_{k=0}^1\,r_k^2\,|S(\theta_k,u_k)\Delta S(\theta_k)|\,,
\]
by the triangular inequality one gets
\begin{eqnarray*}
  |\E(1)\Delta\E_0(1)|\le(1+\s)^2\,\sum_{k=0}^1\,r_k^2\,|S(\theta_k,u_k)\Delta S(\theta_k)|+\Big|(1+\s)\E_0(1)\Delta \E_0(1)\Big|\,.
\end{eqnarray*}
By \cite[Lemma 4]{FigalliMaggiARMA}, if $|\s|<1/2$ and $E\subset B_R\subset\R^n$, then $|E\Delta (1+\s) E|\le C(n)\,R\,|\s|\,P(E)$. Moreover, by scaling, $\E_0(1)\subset B_{C\sqrt{m_1}}$ and $P(\E_0(1))\le C\,\sqrt{m_1}$. Hence,
\[
\Big|(1+\s)\E_0(1)\Delta \E_0(1)\Big|\le C\,m_1\,|\s|\,.
\]
Thus, by $(1+\s)^2\le 9/4$ (recall that $|\s|<1/2$), we conclude
\[
|\E(1)\Delta\E_0(1)|\le C\Big(\sum_{k=0}^1\,r_k^2\,\int_{-\theta_k}^{\theta_k}\,|u_k|+m_1|\s|\Big)
\le C\Big(\sum_{k=0}^1\,r_k^2\theta_k^{1/2}\,\Big(\int_{-\theta_k}^{\theta_k}\,u_k^2\Big)^{1/2}+m_1|\s|\Big)\,,
\]
where \eqref{giovani inesperti} was also taken into account. In conclusion,
\[
|\E(1)\Delta\E_0(1)|^2\le C\Big(\sum_{k=0}^1\, r_k^4\theta_k\,\int_{-\theta_k}^{\theta_k}\,u_k^2+m_1^2\,\s^2\Big)\,.
\]
By arguing similarly with $\E(2)$ in place of $\E(1)$, and since $m_2>m_1$, we obtain \eqref{alphaaa1}.
\end{proof}

The previous results indicate that in order to prove \eqref{global stability double-bubble} on $(\e,\s)$-perturbation (say, in the case $m_2>m_1$) we have to provide a control over
\begin{equation}
  \label{gatto}
  \sum_{k=0}^2\,\int_{-\theta_k}^{\theta_k}u_k^2
\end{equation}
in terms of
\begin{equation}
  \label{topo}
  \sum_{k=0}^2\,\int_{-\theta_k}^{\theta_k}(u_k')^2-u_k^2\,.
\end{equation}
However
\begin{equation}
  \label{fufu}
  \int_{-\theta}^{ \theta}(u')^2-u^2\,,
\end{equation}
is not $L^2$-coercive on $W^{1,2}_0(-\theta,\theta)$, unless $\theta<\pi/2$. Indeed, we easily see that
\[
\inf\Big\{\int_{-\theta}^{\theta}(u')^2:u\in W^{1,2}_0(-\theta,\theta)\,,\int_{-\theta}^{\theta}u^2=1\Big\}=\Big(\frac\pi{2\theta}\Big)^2\,,\qquad\forall \theta>0\,,
\]
so that the best control over $\|u\|_{L^2(-\theta,\theta)}^2$ in terms of $\|u'\|_{L^2(-\theta,\theta)}^2$ is
\begin{eqnarray}
  \label{sharp poincare}
  \int_{-\theta}^{ \theta}(u')^2\ge \Big(\frac\pi{2\theta}\Big)^2\int_{-\theta}^\theta\,u^2\,,\qquad\forall u\in W^{1,2}_0(-\theta,\theta)\,.
\end{eqnarray}
In other words, if $\theta>\pi/2$, then
\[
\inf\Big\{\int_{-\theta}^{ \theta}(u')^2-u^2:u\in W^{1,2}_0(-\theta,\theta)\Big\}=-\infty\,.
\]
Taking into account that $\theta_1$ and $\theta_2$ may possibly range on $(\pi/2,\pi)$, see \eqref{theta intervalli}, we conclude that in order to control \eqref{gatto} in terms of \eqref{topo} we necessarily have to exploit the interaction between the single perturbations $u_k$ through the multiple volume constraints. We now discuss this issue through a careful application of two Poincar\'e-type inequalities. We start by addressing the minimization of \eqref{fufu} under a constraint on the mean value of $u$.

\begin{lemma}\label{stima_inf_2}
If $\theta\in(0,{\pi})$ and $s\in\R$, then
  \begin{eqnarray}\label{eq:stima_inf_2}
    \inf\Big\{\int_{-\theta}^{\theta} (u')^2-u^2:u\in W^{1,2}_0(-\theta,\theta)\,,\int_{-\theta}^{\theta} u=s\Big\}=
    \frac{s^2\,\cos\theta}{2(\sin\theta-\theta\,\cos\theta)}\,.
  \end{eqnarray}
Notice that $\sin\theta-\theta\,\cos\theta$ defines an increasing function on $(0,\pi)$, with values in $(0,\pi)$. Thus the right-hand side of \eqref{eq:stima_inf_2}  decreases from $+\infty$ to $0$ as $\theta\in(0,\pi/2)$, is equal to $0$ for $\theta=\pi/2$, and decreases from $0$ to $-s^2/2\pi$ as $\theta\in(\pi/2,\pi)$.
\end{lemma}

\begin{proof} Given $u\in W_0^{1,2}(-\theta,\theta)$ with $\int_{-\theta}^{\theta}u=s$, let $v(t)=u(t\theta/\pi)$. Thus $v\in W_0^{1,2}(-\pi,\pi)$,
\begin{eqnarray}\label{eq: average-v}
\int_{-\theta}^{\theta}v=s\,\frac{\pi}{\theta}\,,\qquad \int_{-\theta}^\theta (u')^2-u^2=\int_{-\pi}^\pi \frac{\pi}{\theta}\,(v')^2-\frac{\theta}{\pi}\,v^2\,.
\end{eqnarray}
Let $\{\phi_k\}_{k\in\N}\subset L^2(-\pi,\pi)$ be the orthonormal basis of trigonometric functions with $\phi_0=(2\pi)^{-1/2}$, and let $c_k=\int_{-\pi}^\pi v\,\phi_k$ the $k$-th Fourier coefficient of $v$. We have
\begin{eqnarray*}\nonumber
\int_{-\pi}^\pi \frac{\pi}{\theta}\,(v')^2-\frac{\theta}{\pi}\,v^2&=&\left(\frac{\pi}{\theta}-\frac{\theta}{\pi}\right)\int_{-\pi}^\pi (v')^2-\frac{\theta}{\pi}\,\int_{-\pi}^\pi v^2-(v')^2\\
&=& \left(\frac{\pi}{\theta}-\frac{\theta}{\pi}\right)\int_{-\pi}^\pi (v')^2+\frac{\theta}{\pi}\left(\sum_{k=1}^\infty k^2\,c_k^2-\sum_{k=0}^\infty c_k^2\right)\nonumber
\\\label{eq: stima-comp}
&\geq& \left(\frac{\pi}{\theta}-\frac{\theta}{\pi}\right)\int_{-\pi}^\pi (v')^2-\frac{\theta}{\pi}\,c_0^2
\\
&=& \left(\frac{\pi}{\theta}-\frac{\theta}{\pi}\right)\int_{-\pi}^\pi (v')^2-\frac{s^2}{2\theta}\,,
\end{eqnarray*}
where in the last equality we used \eqref{eq: average-v} to compute $c_0$. We have thus proved that
\[
\int_{-\theta}^\theta (u')^2-u^2\ge \bigg(1-\Big(\frac{\theta}{\pi}\Big)^2\bigg)\int_{-\theta}^\theta (u')^2-\frac{1}{2\theta}\Big(\int_{-\theta}^\theta u\Big)^2\,,\qquad\forall u\in W^{1,2}_0(-\theta,\theta)\,,
\]
which immediately lead to prove the existence of minimizers in \eqref{eq:stima_inf_2} by a standard application of the Direct Method. We may thus consider a minimizer $u$ in \eqref{eq:stima_inf_2}, that has to be a smooth solution to the Euler-Lagrange equation
 \begin{equation}\label{euler}
  \left\{\begin{array}
    {l l}
    u''+u=c\,,
    \\
    u(\theta)=u(-\theta)=0\,,
  \end{array}
  \right .
  \end{equation}
for some $c\in\R$. If $\theta=\pi/2$, then $u(t)=\cos(t)$ solves \eqref{euler} (with $c=0$), and, correspondingly, the infimum in \eqref{eq:stima_inf_2} is equal to zero. If, instead, $\theta\ne \pi/2$, then \eqref{euler} has solution
\begin{equation*}
  u(t)=c\,\Big(1-\frac{\cos t}{\cos\theta}\Big)\,,\qquad |t|<\theta\,.
\end{equation*}
A simple computation then gives,
\[
s=\int_{-\theta}^\theta u= 2c\,\Big(\theta-\tan\theta\Big)\,,\qquad\mbox{that is}\qquad c=\frac{s}{2(\theta-\tan\theta)}\,.
\]
Therefore, again by direct computation,
  \begin{eqnarray*}
    \int_{-\theta}^\theta (u')^2-u^2=\frac{-s^2}{2(\theta-\tan\theta)}=\frac{s^2\,\cos\theta}{2(\sin\theta-\theta\,\cos\theta)}\,.
  \end{eqnarray*}
\end{proof}


\begin{lemma}\label{lemma: fuglede}
  For every $\theta\in(0,\pi)$ there exists $M=M(\theta)$ such that, if $u\in W_0^{1,2}(-\theta,\theta)$ with
  \begin{equation}
    \label{fuglede calotte hp}
      \left(\int_{-\theta}^{\theta}u\, \right)^2\le \frac1{M}\,\int_{-\theta}^{\theta}u^2\,,
  \end{equation}
  then
  \begin{equation}
    \label{aleeee}
      \int_{-\theta}^{\theta}(u')^2-u^2\ge\frac14\,\Big(1-\frac{\theta^2}{\pi^2}\Big)\,\int_{-\theta}^{\theta}(u')^2
  +\frac1{2}\,\Big(\frac{\pi^2}{\theta^2}-1\Big)\,\int_{-\theta}^{\theta}\,u^2\,.
  \end{equation}
  A possible value for $M=M(\theta)$ is
  \begin{equation}
    \label{Mtheta}
      M=\frac1{\theta}\,\frac{2\pi^2}{\pi^2-\theta^2}\,.
  \end{equation}
\end{lemma}

\begin{proof}
  Given $u\in W_0^{1,2}(-\theta,\theta)$, define $v\in W_0^{1,2}(-\pi,\pi)$ as $v(t)=u(t\,\theta/\pi)$. By \eqref{fuglede calotte hp},
  \begin{equation}
    \label{fuglede calotte hp su v}
      \left(\int_{-\pi}^\pi v\,\right)^2\le \frac{\pi}{\theta\,M}\,\int_{-\pi}^\pi v^2\,,
  \end{equation}
Let $\phi_k$ and $c_k$ be defined as in the proof of Lemma \ref{stima_inf_2}. For every $\lambda\in(0,1)$ we have
  \begin{eqnarray}\nonumber
 (1-\lambda)\int_{-\theta}^\theta(u')^2-\int_{-\theta}^\theta u^2&=&\frac\pi\theta(1-\lambda)\sum_{k=1}^\infty k^2c_k^2-\frac{\theta}\pi\sum_{k=0}^\infty c_k^2\\
  &\geq&\nonumber \left(\frac\pi\theta(1-\lambda)-\frac\theta\pi\right)\sum_{k=0}^\infty c_k^2-\frac{\pi}{\theta}(1-\lambda)c_0^2\\
  &\geq&
  \frac\pi\theta\left(\frac\pi\theta(1-\lambda)-\frac\theta\pi-\frac{\pi(1-\lambda)}{2\theta^2M}\right)\int_{-\theta}^\theta u^2,
  \label{lemma: stima-fourier-c0}
  \end{eqnarray}
  where we have estimated $c_0$ thanks to \eqref{fuglede calotte hp} as follows,
  \begin{eqnarray*}
  c_0^2=\frac{1}{2\pi}\left(\int_{-\pi}^{\pi}v\,\right)^2\leq\frac{1}{2\theta M}\int_{-\pi}^{\pi}v^2=\frac{\pi}{2\theta^2 M}\int_{-\theta}^{\theta}u^2\,.
  \end{eqnarray*}
  Let us now rearrange \eqref{lemma: stima-fourier-c0} as
  \[
  \int_{-\theta}^\theta(u')^2- u^2
  \ge\l\int_{-\theta}^\theta(u')^2+
  \bigg( \frac{\pi^2}{\theta^2}\,\Big(1-\frac{1}{2\theta\,M}\Big) (1-\lambda)-1\bigg)\,\int_{-\theta}^\theta u^2\,.
  \]
  We prove \eqref{aleeee} by choosing $M$ as in \eqref{Mtheta}, by setting
  \begin{eqnarray*}
  \lambda=\frac14\Big(1-\frac{\theta^2}{\pi^2}\Big)=\frac14\frac{\theta^2}{\pi^2}\Big(\frac{\pi^2}{\theta^2}-1\Big)\,,
  \end{eqnarray*}
  and finally noticing that
  \[
  \frac{\pi^2}{\theta^2}\,\Big(1-\frac{1}{2\theta\,M}\Big) (1-\lambda)-1\ge
  \frac{\pi^2}{\theta^2}\,-1-\frac{\pi^2}{\theta^2}\Big(\l+\frac1{2\theta M}\Big)=\frac12\Big(\frac{\pi^2}{\theta^2}\,-1\Big)\,.
  \]
\end{proof}

We finally prove Theorem \ref{thm main 2} in the case of $(\e,\s)$-perturbations.

\begin{theorem}\label{thm Fuglede double-bubble}
  For every $m_2\ge m_1>0$, there exist positive constants $\e_1$, $\s_1$, and $\k_1$ (depending on $m_1/m_2$ only) with the following property. If  $\E$ is an $(\e,\s)$-perturbation of $\E_0$ with $\vol(\E_0)=(m_1,m_2)$, and if $\e<\e_1$ and $|\s|<\s_1$, then, in the case $m_2>m_1$
  \begin{equation}
    \label{stab}
      P(\E)-P(\E_0)\ge\k_1\,\Big(\s^2+\sum_{k=0}^2\,r_k\int_{-\theta_k}^{\theta_k}u_k^2\Big)\,,
  \end{equation}
  while, in the case $m_2=m_1$ (and $r_2=r_1=r$),
  \begin{equation}
    \label{stab2}
      P(\E)-P(\E_0)\ge\k_1\,\Big(\s^2+r\,\int_{-\sqrt{3}/2}^{\sqrt{3}/2}v_0^2+\sum_{k=1}^2\,r\,\int_{-2\pi/3}^{2\pi/3} u_k^2\Big)\,.
  \end{equation}
  In both cases, by Lemma \ref{lemma upper}, there exists $\k_1^*$ depending on $m_1$ and $m_2$ such that
  \begin{equation}
    P(\E)\ge P(\E_0)\big\{1+\k_1^*\a(\E)^2\big\}\,.
  \end{equation}
\end{theorem}

\begin{proof} {\it Step one:} Let $\theta\in(0,\pi)$, and let $M(\theta)$ be as in \eqref{Mtheta}. We notice that for every $\theta\in(0,\pi)$ there exists $\e(\theta)>0$ such that if
\[
\|u\|_{C^0(-\theta,\theta)}\le\e(\theta)\,,\qquad \Big(\int_{-\theta}^{\theta}u+\frac{u^2}2\Big)^2\le\frac{1}{2\,M(\theta)}\int_{-\theta}^{\theta}u^2\,,
\]
then
\[
\Big(\int_{-\theta}^{\theta}u\Big)^2\le\frac{1}{M(\theta)}\int_{-\theta}^{\theta}u^2\,.
\]
In the rest of the proof, given $m_1$ and $m_2$, and thus fixed $\theta_1$ and $\theta_2$ according to \eqref{r yeah 1} and \eqref{r yeah 2}, we shall assume to work with $(\e,\s)$-perturbations of $\E_0$ with $\e<\min\{\e(\theta_1),\e(\theta_2)\}$.

\medskip

\noindent {\it Step two:} We start considering the case $m_2>m_1$. If $\E$ is an $(\e,\s)$-perturbation of $\E_0$ with functions $u_0$, $u_1$, and $u_2$, then, for $t>0$, $t\,\E$ is an $(\e,\s)$-perturbation of $t\,\E_0$ with the same functions $u_0$, $u_1$, and $u_2$. Therefore, without loss of generality, in the following we may assume that $r_2=1$. For the sake of symmetry (and, thus, of clarity) we shall keep writing $r_2$ in place of $1$ in the following formulas, until we exploit this scaling assumption. Let us now set
\[
I_k=\int_{-\theta_k}^{\theta_k}u_k+\frac{u_k^2}2\,,\qquad k=0,1,2\,,
\]
so that the volume constraints \eqref{gippo1} and \eqref{gippo2} take the form
\begin{eqnarray}\label{angelino1}
  I_0&=&-\Big(\frac{r_1}{r_0}\Big)^2\,I_1+\frac{m_1}{r_0^2}\Big(\frac1{(1+\s)^2}-1\Big)\,,
  \\\label{angelino2}
  I_0&=&\Big(\frac{r_2}{r_0}\Big)^2\,I_2-\frac{m_2}{r_0^2}\Big(\frac1{(1+\s)^2}-1\Big)\,.
\end{eqnarray}
Multiplying \eqref{angelino1} by $m_2/(m_1+m_2)$, \eqref{angelino2} by $m_1/(m_1+m_2)$, and then adding up, we find
\begin{equation}\label{Izero12}
I_0 = \frac{m_1}{m_1+m_2}\,\Big(\frac{r_2}{r_0}\Big)^2\,I_2
- \frac{m_2}{m_1+m_2}\,\Big(\frac{r_1}{r_0}\Big)^2\,I_1\,.
\end{equation}
Similarly, multiplying both \eqref{angelino1} and \eqref{angelino2} by $r_0^2$, and then subtracting the resulting identities, we come to
$r_1^2\,I_1+r_2^2\,I_2=(m_1+m_2)((1+\s)^{-2}-1)$, which gives
\begin{equation}\label{phis12}
\s^2+\O(|\s|^3)= \frac{(r_{1}^{2}I_1 + r_{2}^{2}I_2)^2}{4(m_1+m_2)^2}\,.
\end{equation}
By \eqref{phis12} we deduce that
\begin{equation}
\label{sigma12}
\sigma^{2}+\O(|\s|^3) \leq \frac{r_{1}^{4}I_1^{2} + r_{2}^{4}I_2^{2}}{2(m_1+m_2)^{2}}+\e\,\O(\|u\|_{L^2}^2)\,,
\end{equation}
and, since $I_k \leq C\int_{-\theta_k}^{\theta_k}u_k^2$,  that $|\sigma|=\O(\|u\|_{L^2})$. (This is a reflection of the fact that if the $u_k$'s are all zero, then, by the volume constraint, we necessarily have $\s=0$.) Thus \eqref{stima deficit 1} gives
\begin{equation}\label{stimadef-op}
2\,\frac{P(\E)-P(\E_0)}{1+\s} = \sum_{k=0}^{2} r_{k}\int_{-\theta_k}^{\theta_k} (u'_{k})^{2} - u_{k}^{2} + P(\E_0)\,\sigma^{2} +(\e+|\s|)\,\O(\|u\|_{W^{1,2}}^2)\,.
\end{equation}
We now claim that, for a suitable constant $C$ (depending on $\E_0$) we have
\begin{eqnarray}\label{ai}
C\,(P(\E)-P(\E_0)) \geq r_1\,I_1^2+r_2\,I_2^2+(\e+|\s|)\,\O(\|u\|_{W^{1,2}}^2)\,.
\end{eqnarray}
To this end, let us set for the sake of brevity
\begin{equation}
  \label{gthetaaaaaaaaaaa}
g(\theta)=\frac{\cos\theta}{2(\sin\theta-\theta\cos\theta)}\,,\qquad 0<\theta<\pi\,.
\end{equation}
By Lemma \ref{stima_inf_2}, for $k=0,1,2$ we have
\begin{equation}\label{Dk-op}
\int_{-\theta_k}^{\theta_k} (u'_{k})^{2} - u_{k}^{2} \geq g(\theta_k)\,\Big(I_k- \int_{-\theta_k}^{\theta_k} \frac{u_{k}^{2}}2\Big)^{2} = g(\theta_k)\,I_k^{2}+ \e\,\O(\|u\|_{L^2}^2)\,,
\end{equation}
and thus, by inserting \eqref{phis12} and \eqref{Dk-op} into \eqref{stimadef-op},
\begin{eqnarray}\label{stimadef-op2}
2\,\frac{P(\E)-P(\E_0)}{1+\s} &\geq & \sum_{k=0}^{2}r_k\,g(\theta_k) I_k^{2} + \frac{P(\E_0)(r_{1}^{2} I_1 + r_{2}^{2}I_2)^{2}}{4(m_1+m_2)^{2}} +(\e+|\s|)\,\O(\|u\|_{W^{1,2}}^2)\nonumber\\
& = & \b_1\,r_1\,I_1^2 + \b_2\,r_2\,I_2^2 + 2\b_3\, \sqrt{r_1r_2}\,I_1 \,I_2  +(\e+|\s|)\,\O(\|u\|_{W^{1,2}}^2)\,.
\end{eqnarray}
Here, by taking into account \eqref{Izero12}, we have set
\begin{eqnarray}\label{b1}
\b_{1} &=& g(\theta_0)\,\frac{r_{1}^3}{r_0^3}\,\frac{m_2^{2}}{(m_1+m_2)^{2}} + g(\theta_1) + \frac{r_{1}^3}4\,\frac{P(\E_0)}{(m_1+m_2)^{2}}\,,
\\\label{b2}
\b_{2} &=& g(\theta_0)\,\frac{r_{2}^{3}}{r_0^3}\,\frac{m_1^{2}}{(m_1+m_2)^{2}} + g(\theta_{2})+ \frac{r_{2}^3}{4}\,\frac{P(\E_0)}{(m_1+m_2)^{2}}\,,
\\\label{b3}
\b_{3} &=& -g(\theta_0)\,\frac{r_1^{3/2}\,r_2^{3/2}}{r_0^3}\,\frac{m_1\,m_2}{(m_1+m_2)^2}+\frac{r_{1}^{3/2}r_{2}^{3/2}}4\,\frac{P(\E_0)}{(m_1+m_2)^{2}}\,.
\end{eqnarray}
The quadratic form in $(\sqrt{r_1}\,I_1,\sqrt{r_2}\,I_2)$ on the right-hand side \eqref{stimadef-op2} is coercive: indeed, it suffices to show the existence of $\beta_*>0$ (depending on $m_1/m_2$ only) such that
\begin{equation}
  \label{betabasso}
  \min\{\beta_1,\beta_1\beta_2-\beta_3^2\}\ge\beta_*\,.
\end{equation}
To this end, let us note that, having set $r_2=1$, it turns out that $r_0$, $\theta_0$, $\theta_1$, $\theta_2$, $m_1$, and $m_2$ are all explicit functions of $r_1\in(0,1)$ according to equations \eqref{r2}, \eqref{r yeah}, \eqref{r yeah 1}, \eqref{r yeah 2}, \eqref{formulam1}, and \eqref{formulam2}. Correspondingly, the coefficients $\b_k$ can be easily expressed as functions of $r_1\in(0,1)$, and the validity of \eqref{betabasso} can be deduced by a numerical plot; see Figure \ref{fig: beta}.
\begin{figure}
 \epsfig{file=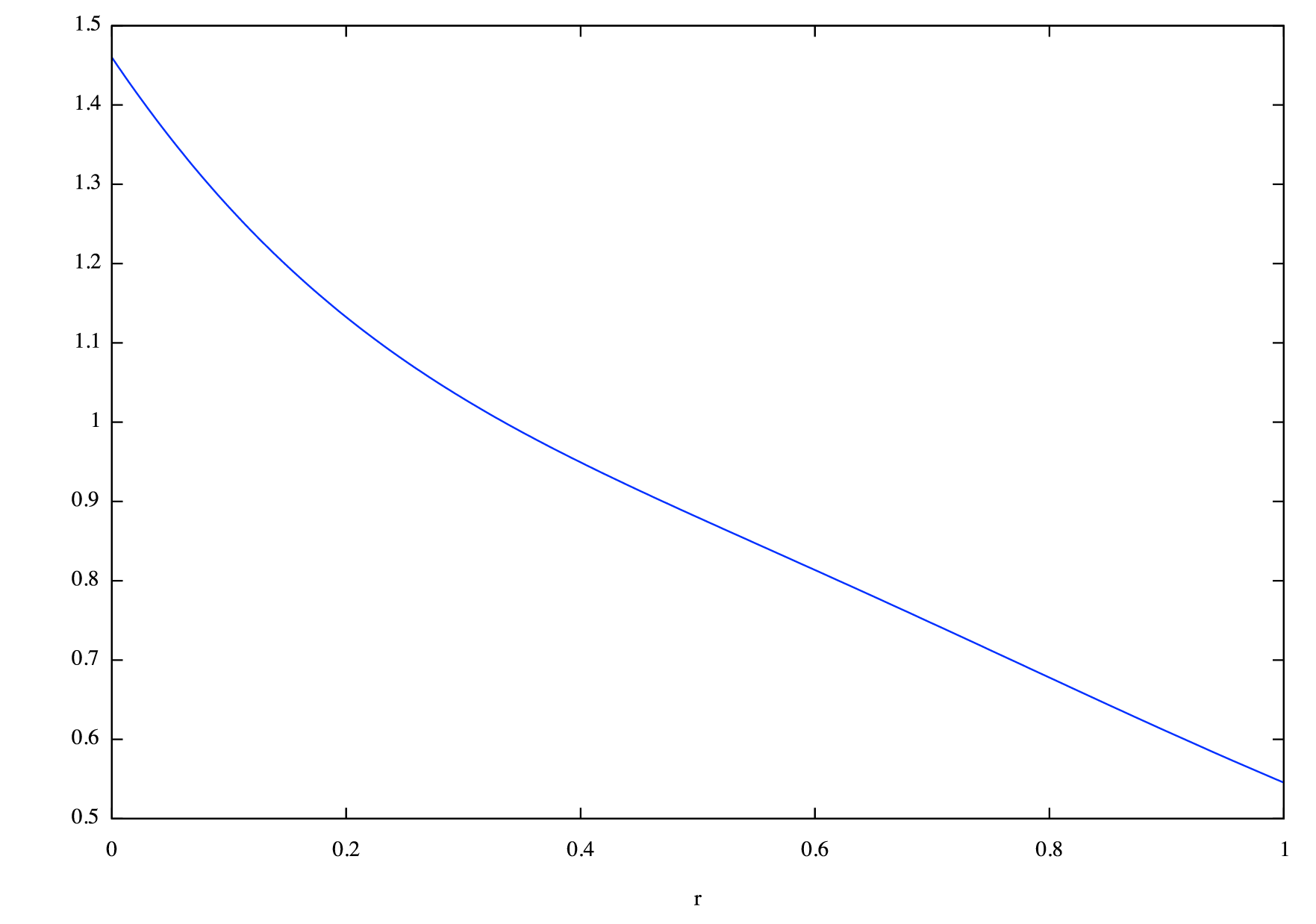, height=4.2cm} \hspace{0.4cm} \epsfig{file=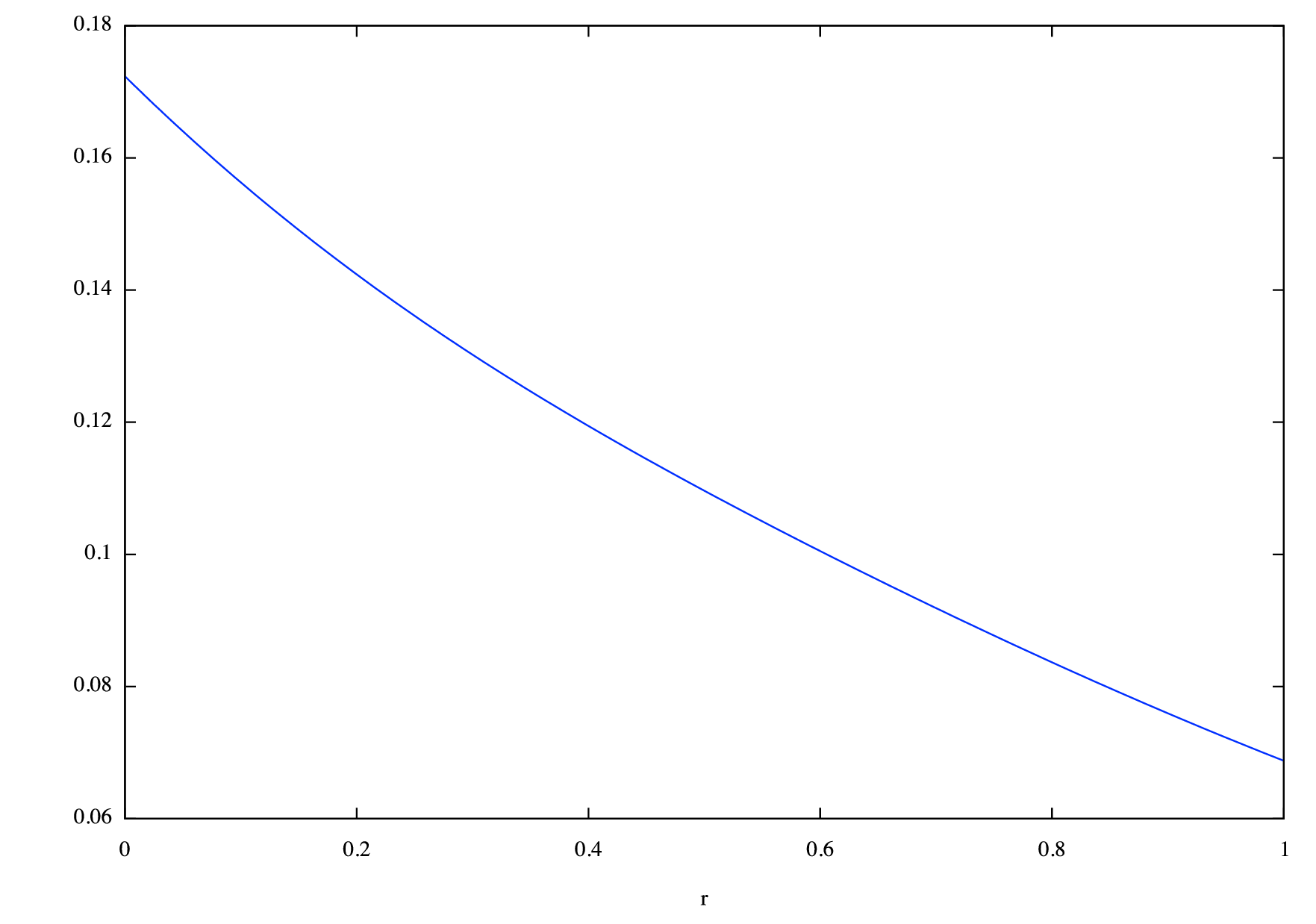, height=4.2cm}\caption{{\small Plotting of $\beta_1(r)$ (left) and of $(\beta_1(r)\beta_2(r)-\beta_3(r)^2)/r$ (right) for $r\in(0,1)$. In particular, $\beta_1(r)\beta_2(r)-\beta_3(r)^2\approx r$ for $r$ small. The plots have been drawn by Maxima v.5.28.0 (http://maxima.sourceforge.net) starting from equations $r_2=1$, $r_1=r\in(0,1)$, \eqref{r2}, \eqref{r yeah}, \eqref{r yeah 1}, \eqref{r yeah 2}, \eqref{gippo rule}, \eqref{formulam1}, \eqref{formulam2}, \eqref{gthetaaaaaaaaaaa}, \eqref{b1}, \eqref{b2}, and \eqref{b3}.}}\label{fig: beta}
\end{figure}
As a consequence of \eqref{betabasso}, and up to decrease the value of $\beta_*$, we find
\[
\beta_1\,r_1\,I_1^2 + \b_2\,r_2\,I_2^2 + 2\b_3\, \sqrt{r_1r_2}\,I_1\, I_2\geq \beta_* (r_1\,I_1^2+r_2\,I_2^2)\,.
\]
We combine this inequality with \eqref{stimadef-op2} to prove \eqref{ai}, as claimed. Now, by \eqref{sigma12} and \eqref{ai},
\begin{eqnarray}\label{aii}
C\,(P(\E)-P(\E_0)) \geq \s^2+r_1\,I_1^2+r_2\,I_2^2+(\e+|\s|)\,\O(\|u\|_{W^{1,2}}^2)\,.
\end{eqnarray}
By the choice of $\e$ performed in step one, we now notice that, if for some $k=1,2$ we have
\[
I_k^2\leq  \frac{1}{2M(\theta_k)}\,\int_{-\theta_k}^{\theta_k}u_k^2\,,
\]
then, by Lemma \ref{lemma: fuglede},
\begin{equation}
  \label{ho1}
  \int_{-\theta_k}^{\theta_k} (u'_{k})^{2} - u_{k}^{2}\ge
\frac{1}4\,\Big(1-\frac{\theta_k^2}{\pi^2}\Big)\,\int_{-\theta_k}^{\theta_k}(u_k')^2
  +\frac{1}{2}\,\Big(\frac{\pi^2}{\theta_k^2}-1\Big)\,\int_{-\theta_k}^{\theta_k}\,u_k^2\,.
\end{equation}
Therefore, for $k=1,2$, either \eqref{ho1} holds true, or
\begin{equation}
  \label{ho2}
  I_k^2\ge \frac{1}{2M(\theta_k)}\,\int_{-\theta_k}^{\theta_k}u_k^2\,.
\end{equation}
Concerning $u_0$, let us notice that, by the sharp Poincar\'e inequality \eqref{sharp poincare}, and since $\theta_0<\pi/3$,
\[
\int_{-\theta_0}^{\theta_0} (u'_0)^{2}\ge\Big(\frac\pi{2\theta_0}\Big)^2\int_{-\theta_0}^{\theta_0} u_0^{2}\ge \frac94\,\int_{-\theta_0}^{\theta_0} u_0^{2}\,,
\]
which gives
\begin{equation}
  \label{ho3}
  \int_{-\theta_0}^{\theta_0} (u'_0)^{2}-u_0^2\ge\frac13\int_{-\theta_0}^{\theta_0} (u'_0)^2+\Big(\frac32-1\Big)\int_{-\theta_0}^{\theta_0}u_0^2\,.
\end{equation}
We are now going to use \eqref{ho1}, \eqref{ho2}, and \eqref{ho3} together with \eqref{aii} to prove that, for some constant $C$ depending on $\E_0$, we always have
\begin{equation}
  \label{heyman2}
  C\,(P(\E)-P(\E_0))\ge\s^2+\sum_{k=0}^2r_k\int_{-\theta_k}^{\theta_k}(u_k')^2+u_k^2\,.
\end{equation}
We divide the argument in three cases:

\medskip

\noindent {\it Case one}: We assume that \eqref{ho1} holds true for $k=1,2$. By this assumption, \eqref{stimadef-op}, and \eqref{ho3},
\begin{equation}
  \label{heyman}
    C\,(P(\E)-P(\E_0))\ge\s^2+\sum_{k=0}^{2} r_{k}\int_{-\theta_k}^{\theta_k} (u'_{k})^{2} + u_{k}^{2}+(\e+|\s|)\,\O(\|u\|_{W^{1,2}}^2)\,,
\end{equation}
from which \eqref{heyman2} is easily proved.

\medskip

\noindent {\it Case two}: We assume that \eqref{ho2} holds true for $k=1,2$. In this case, by \eqref{stimadef-op} we obtain
\begin{eqnarray}\nonumber
  2\,\frac{P(\E)-P(\E_0)}{1+\s}&\ge&\tau\,\Big(\sum_{k=0}^{2} r_{k}\int_{-\theta_k}^{\theta_k} (u'_{k})^{2} - u_{k}^{2}\Big)+(1-\tau)\,2\,\frac{P(\E)-P(\E_0)}{1+\s}
  \\\nonumber
  &&+(\e+|\s|)\,\O(\|u\|_{W^{1,2}}^2)
  \\\nonumber
  \mbox{\small{(by \eqref{ho3})}}\qquad&\ge&
  \tau\Big(\frac{r_0}3\int_{-\theta_0}^{\theta_0} (u'_0)^2+\frac{r_0}2\int_{-\theta_0}^{\theta_0}u_0^2 \Big)+\tau\,\sum_{k=1}^{2} r_{k}\int_{-\theta_k}^{\theta_k} (u'_{k})^{2}-u_k^2
  \\\nonumber
  \mbox{\small{(by \eqref{aii})}}\qquad&&+\frac{1-\tau}C\,\Big(\s^2+r_1\,I_1^2+r_2\,I_2^2\Big)+(\e+|\s|)\,\O(\|u\|_{W^{1,2}}^2)
  \\\nonumber
  &\ge&
  \tau\Big(\frac{r_0}3\int_{-\theta_0}^{\theta_0} (u'_0)^2+\frac{r_0}2\int_{-\theta_0}^{\theta_0}u_0^2 \Big)+\tau\,\sum_{k=1}^{2} r_{k}\int_{-\theta_k}^{\theta_k} (u'_{k})^{2}-u_k^2
  \\\nonumber
  \mbox{\small{(by \eqref{ho2} for $k=1,2$)}}\qquad&&+\frac{1-\tau}C\,\Big(\s^2+\sum_{k=1}^2\frac{r_k}{2M(\theta_k)}\,\int_{-\theta_k}^{\theta_k}u_k^2\Big)
  +(\e+|\s|)\,\O(\|u\|_{W^{1,2}}^2)
   \\\nonumber
  &\ge&
  \tau\Big(\frac{r_0}3\int_{-\theta_0}^{\theta_0} (u'_0)^2+\frac{r_0}2\int_{-\theta_0}^{\theta_0}u_0^2 \Big)+\tau\,\sum_{k=1}^{2} r_{k}\int_{-\theta_k}^{\theta_k} (u'_{k})^{2}
  \\\nonumber
 &&+\frac{1-\tau}{2C}\,\Big(\s^2+\sum_{k=1}^2\frac{r_k}{2M(\theta_k)}\,\int_{-\theta_k}^{\theta_k}u_k^2\Big)
 +(\e+|\s|)\,\O(\|u\|_{W^{1,2}}^2)\,,
\end{eqnarray}
where in the last inequality we have absorbed the negative terms in $u_k^2$, $k=1,2$, by choosing $\tau$ so small to have
\[
\tau\le\frac{1-\tau}{4\,C}\min_{k=1,2}\,\frac{1}{M(\theta_k)}\,.
\]
We have thus proved \eqref{heyman}, and thus \eqref{heyman2}, up to suitably choose $\e$ and $C$.

\medskip

\noindent {\it Case three:} We assume that \eqref{ho1} holds true for $k=1$, while \eqref{ho2} holds true for $k=2$. By arguing as in case two we find, for any $\tau\in(0,1)$,
\begin{eqnarray*}
  2\,\frac{P(\E)-P(\E_0)}{1+\s}&\ge&
  \tau\Big(\frac{r_0}3\int_{-\theta_0}^{\theta_0} (u'_0)^2+\frac{r_0}2\int_{-\theta_0}^{\theta_0}u_0^2 \Big)+\tau\,\sum_{k=1}^{2} r_{k}\int_{-\theta_k}^{\theta_k} (u'_{k})^{2}-u_k^2
  \\\nonumber
  &&+\frac{1-\tau}C\,\Big(\s^2+r_1\,I_1^2+r_2\,I_2^2\Big)
  +(\e+|\s|)\,\O(\|u\|_{W^{1,2}}^2)\,.
\end{eqnarray*}
By using \eqref{ho1} for $k=1$ and \eqref{ho2} for $k=2$, and discarding some positive terms, we find
\begin{eqnarray*}
  2\,\frac{P(\E)-P(\E_0)}{1+\s}&\ge&
  \tau\,c\,\bigg(r_0\,\int_{-\theta_0}^{\theta_0} \Big((u'_0)^2+u_0^2\Big)+r_{1}\,\int_{-\theta_1}^{\theta_1} \Big((u'_1)^{2}+u_1^2\Big)+r_2\,\int_{-\theta_2}^{\theta_2} (u'_2)^{2}\bigg)
  \\\nonumber
  &&+\frac{1-\tau}C\,\Big(\s^2+\frac{r_2}{2M(\theta_2)}\,\int_{-\theta_2}^{\theta_2}u_2^2\Big)-\tau\,r_2\int_{-\theta_2}^{\theta_2} u_2^2+
 (\e+|\s|)\,\O(\|u\|_{W^{1,2}}^2)\,,
\end{eqnarray*}
for some positive constant $c$ depending on $\E_0$. As in case two, we may choose $\tau$ small enough to have the negative term in $u_2^2$ absorbed by its positive counterpart, and come to prove \eqref{heyman}. Finally, when \eqref{ho1} holds true for $k=2$ and \eqref{ho2} holds true for $k=1$ (note that, formally, this is a fourth different case, as $m_2>m_1$), then we just repeat the very same argument. Summarizing, we have proved the validity of \eqref{heyman2}, which of course implies \eqref{stab}. The theorem is proved in the case $m_2>m_1$.

\medskip

\noindent {\it Step three:} We now address the case $m_2=m_1$. In this case we set $r=r_1=r_2$, $m=m_1=m_2$, and $\theta=\theta_1=\theta_2=2\pi/3$. Once again, up to scaling, we may assume that $r=1$, so that
\[
m=\frac{2\pi}3+\frac{\sqrt{3}}{4}\,,\qquad P(\E_0)=4\,m=\frac{8\,\pi}3+\sqrt{3}\,.
\]
The volume constraints now take the form
\begin{eqnarray*}
  \Big((1+\s)^{-2}-1\Big)\,m=I_1+\int_{-\sqrt{3}/2}^{\sqrt{3}/2}v_0=I_2-\int_{-\sqrt{3}/2}^{\sqrt{3}/2}v_0\,,
\end{eqnarray*}
so that, by arguing as in step one, we find, in analogy to \eqref{Izero12} and \eqref{phis12},
\begin{equation}
  \label{febbre1}
    \int_{-\sqrt{3}/2}^{\sqrt{3}/2}v_0=\frac{I_2-I_1}2\,,\qquad
  \s^2+\O(|\s|^3)=\frac{(I_1+I_2)^2}{4\,m^2}\,.
\end{equation}
By Lemma \ref{stima_inf_2} we have \eqref{Dk-op} for $k=1,2$, and, similarly,
\begin{eqnarray}\label{febbre2}
\int_{-\sqrt{3}/2}^{\sqrt{3}/2}(v_0')^2\ge \int_{-\sqrt{3}/2}^{\sqrt{3}/2}v_0^2+g\Big(\frac{\sqrt{3}}2\Big)\,\Big(\int_{-\sqrt{3}/2}^{\sqrt{3}/2}v_0\Big)^2
=\int_{-\sqrt{3}/2}^{\sqrt{3}/2}v_0^2+g\Big(\frac{\sqrt{3}}2\Big)\,\frac{(I_2-I_1)^2}4\,.
\end{eqnarray}
(Notice that $\sqrt{3}/2<\pi/2$, thus $g(\sqrt{3}/2)$ is positive.) By \eqref{febbre1} and \eqref{febbre2}, and since $|\s|=\O(\|u\|_{L^2}^2)$, from  \eqref{stima deficit 1bis} we deduce
 \begin{eqnarray*}\nonumber
  2\,\frac{P(\E)-P(\E_0)}{1+\s}
  &=&\int_{-\,\sqrt{3}/2}^{\,\sqrt{3}/2}(v'_0)^2+
  \sum_{k=1}^2\,\int_{-2\pi/3}^{2\pi/3}(u'_k)^2-u_k^2+\frac{\s^2}{2}\,P(\E_0)
  +(\e+|\s|)\,\O(\|u\|_{W^{1,2}}^2)
  \\
  &\ge&\int_{-\sqrt{3}/2}^{\sqrt{3}/2}v_0^2+g\Big(\frac{\sqrt{3}}2\Big)\,\frac{(I_2-I_1)^2}4+g\Big(\frac{2\pi}3\Big)(I_1^2+I_2^2)+\frac{(I_1+I_2)^2}{2\,m}
  \\&&+(\e+|\s|)\,\O(\|u\|_{W^{1,2}}^2)
  \\
  &\ge&\int_{-\sqrt{3}/2}^{\sqrt{3}/2}v_0^2+\a_1\,I_1^2+\a_2\,I_2^2+2\a_3\,I_1\,I_2\,,+(\e+|\s|)\,\O(\|u\|_{W^{1,2}}^2)\,,
  \end{eqnarray*}
 provided we set
 \begin{eqnarray*}
   &&\a_1=\a_2=\frac14\,g\Big(\frac{\sqrt{3}}2\Big)+g\Big(\frac{2\pi}3\Big)+\frac1{2m}
   \\
   &&\a_3=-\frac14\,g\Big(\frac{\sqrt{3}}2\Big)+\frac1{2m}\,.
 \end{eqnarray*}
 By direct evaluation we see that $\a_1>0$ and $\a_1\a_2-\a_3^2>0$. Therefore there exists $\a_*>0$ such that $\a_1\,I_1^2+\a_2\,I_2^2+2\a_3\,I_1\,I_2\ge\a_*(I_1^2+I_2^2)$, and thus
 \begin{eqnarray}\label{aiiiii}
  2\,\frac{P(\E)-P(\E_0)}{1+\s}
  \ge\int_{-\sqrt{3}/2}^{\sqrt{3}/2}v_0^2+\a_*\,(I_1^2+I_2^2)+(\e+|\s|)\,\O(\|u\|_{W^{1,2}}^2)\,.
  \end{eqnarray}
 We conclude the proof exactly as in step two, with \eqref{aiiiii} playing the role of \eqref{ai}, and with
\begin{equation}
  \label{ho4}
  \int_{-\sqrt{3}/2}^{\sqrt{3}/2} (v'_0)^{2}\ge \frac12\int_{-\sqrt{3}/2}^{\sqrt{3}/2} (v'_0)^{2}
  + v_0^{2}
\end{equation}
playing the role of \eqref{ho3}. (Note that \eqref{ho4} follows trivially from \eqref{febbre2}.) This completes the proof of Theorem \ref{thm Fuglede double-bubble}.
\end{proof}

\appendix

\section{The qualitative stability theorem and a selection principle}\label{appendix soft and select} Here we prove a qualitative stability theorem (Theorem \ref{thm asymptotic stability}) and a selection principle for quantitative stability inequalities (Theorem \ref{thm selection principle part one}) on isoperimetric $N$-clusters in $\R^n$ with $n$ and $N$ arbitrary. These results are not entirely standard because of some compactness issues that need to be handled under a multiple volumes constraint. Such compactness issues are usually simpler to address in dimension $n=2$ (because perimeter controls diameter on indecomposable sets of finite perimeter), and in this paper we only need the above results in the case $N=n=2$. However, Theorem \ref{thm asymptotic stability} is interesting in itself and it is useful knowing its validity in the general case. Theorem \ref{thm selection principle part one}, although of course of more technical nature, should still reveal useful in addressing the quantitative stability problem for double-bubbles in higher dimensions. Moreover, the simplifications one has setting $n=2$ seem not that significant, at least if one exploits the arguments we know to prove these results. For these reasons we have decided to prove these theorems in full generality.

The setting considered in this appendix will be as follows. Given a $N$-cluster $\E_0$ in $\R^n$ one says that $\E_0$ is an {\it isoperimetric cluster} if $P(\E_0)\le P(\E)$ whenever $\vol(\E)=\vol(\E_0)$, and that $\E_0$ is {\it uniquely minimizing} if $P(\E)=P(\E_0)$ and $\vol(\E)=\vol(\E_0)$ imply the existence of an isometry $f:\R^n\to\R^n$ such that $f(\E)=\E_0$, where we have set $f(\E)(h)=f(\E(h))$ for every $h=1,...,N$. For a uniquely minimizing isoperimetric cluster $\E_0$ in $\R^n$, we set
\begin{eqnarray*}
  \M_{0}&=&\big\{\E:\ \E \text{ is an $N$-cluster, }\vol(\E) = \vol(\E_{0})\big\}\,,
  \\
  \de(\E)&=&P(\E)-P(\E_0)\,,
  \\
  \a(\E)&=&\inf\big\{\d(\E,f(\E_0)):\mbox{$f:\R^n\to\R^n$ is an isometry}\big\}\,,
\end{eqnarray*}
where $\d(\E,\F)=(1/2)\,\sum_{h=0}^N|\E(h)\Delta\F(h)|$. Note that if $\E\in\M_0$, then $\de(\E)$ and $\a(\E)$ are both positive unless $\E$ is isometric to $\E_0$. In analogy with the case $N=1$ \cite{fuscomaggipratelli}, one may ask about the validity of a quantitative stability inequality of the form
\begin{equation}
    \label{global stability inequality delta alpha}
    \de(\E)\ge\k\,\a(\E)^2\,,\qquad\forall \E\in\M_0\,,
\end{equation}
for some $\k>0$. As a first step in this direction, one wants to prove the following theorem.

\begin{theorem}
  \label{thm asymptotic stability} If $\E_0$ is a uniquely minimizing isoperimetric cluster in $\R^n$, $n\ge 2$, then for every $\eta>0$ there exists $\de>0$ such that  if $\vol(\E)=\vol(\E_0)$ and $P(\E)\le P(\E_0)+\de$, then $\a(\E)\le\eta$.
\end{theorem}

Once Theorem \ref{thm asymptotic stability} is proved, and following the approach proposed in \cite{CicaleseLeonardi} to address  \eqref{global stability inequality delta alpha} in the case $N=1$, one notices that by a simple contradiction argument \eqref{global stability inequality delta alpha} is equivalent to showing that $\k(\E_0)>0$, where we have set
\begin{equation}
  \label{kappa ottimale}
  \kappa(\E_0)=\inf\big\{\liminf_{k\to\infty}\frac{\de(\E_k)}{\a(\E_k)^2}:\{\E_k\}_{k\in\N}\subset\M_0\,,\a(\E_k)>0\,,\E_k\to\E_0\big\}\,.
\end{equation}
By applying a selection principle to minimizing sequences in  \eqref{kappa ottimale}, one ends up reducing the proof of \eqref{global stability inequality delta alpha} to the case when $\E$ is a $(\Lambda,r_0)$-minimizing cluster in $\R^n$ for some $\Lambda\ge0$ and $r_0>0$ depending on $\E_0$ only. In the case $N=1$, as shown in \cite{CicaleseLeonardi}, this reduction allows one to complete the proof of \eqref{global stability inequality delta alpha} quite easily thanks to a decomposition in spherical harmonics originally introduced by Fuglede \cite{Fuglede}. At the same time, as shown in this paper, this strategy works to prove \eqref{global stability inequality delta alpha} when $N=n=2$. It thus seems interesting to know that one can always attack \eqref{global stability inequality delta alpha} from this angle. More precisely, we have the following result.

\begin{theorem}\label{thm selection principle part one}
 If $\E_0$ is a uniquely minimizing isoperimetric cluster in $\R^n$ with $\k(\E_0)<\infty$, then there exist positive constants $\Lambda$, $r_0$, and $R_0$ and a sequence of $(\Lambda,r_0)$-minimizing clusters $\{\E_k\}_{k\in\N}\subset\M_0$ with
 \[
 \inf_{k\in\N}\a(\E_k)>0\,,\qquad \lim_{k\to\infty}\d(\E_k,\E_0)=0\,,\qquad \lim_{k\to\infty}\frac{\de(\E_k)}{\a(\E_k)^2}=\k(\E_0)\,.
 \]
 Moreover, $\E_k(h)\subset B_{R_0}$ for every $h=1,...,N$, and each $\E_k$ satisfies the global, volume-constrained minimality property
     \begin{equation}\label{C-k}
     P(\E_{k}) \leq P(\F) + 3\,\sqrt{\a(\E_k)}\, \d(\F,\E_{k})\,,\qquad\forall \F\in\M_0\,.
     \end{equation}
\end{theorem}

\begin{remark}
  {\rm The assumption $\k(\E_0)<\infty$ is essentially equivalent to showing the existence of a one-parameter family of clusters $\{\E_t\}_{|t|<\e}$ with $\vol(\E_t)=\vol(\E_0)$, $\a(\E_t)>0$, $P(\E_t)-P(\E_0)\le C\,t^2$, and $\a(\E_t)\ge |t|/C$ for every $|t|<\e$. By Theorem \ref{thm volume fixing} below, it is not difficult to define $\E_t$ satisfying the first three conditions: what is not immediate, however, is proving that $\a(\E_t)\ge |t|/C$. When $N=1$ or $N=2$ (see section \ref{section select principle} for the latter case) one can easily address this point by exploiting the symmetries of the corresponding isoperimetric clusters (balls or standard double-bubbles). For general $N$ one does not expect to have symmetry properties or to explicitly characterize isoperimetric clusters. Nevertheless, it is always true that $\k(\E_0)<\infty$. We shall not further discuss this issue here.}
\end{remark}

We now turn to prove Theorem \ref{thm asymptotic stability} and Theorem \ref{thm selection principle part one}. As explained the issue is the lack of global compactness, and thus of the possible loss of volume at infinity. This can be fixed by exploiting an argument similar to the one used in Almgren's proof \cite{Almgren76} of the existence of isoperimetric clusters for every given volume vector, see also \cite[Chapter 29]{maggiBOOK}. Almgren's argument uses truncations and translations of pieces of the quasi-isoperimetric clusters, so what one needs to do is taking track of what happens to $\a(\E)$ under these operations. The following theorem is a key tool in implementing this strategy. It is a variant of \cite[Proposition VI.12]{Almgren76}, see also \cite[Corollary 29.17]{maggiBOOK}. The necessary modifications with respect to \cite[Corollary 29.17]{maggiBOOK} are described in \cite[Appendix B]{CiLeMaIC1}, so that we omit to give a detailed proof in here.

\begin{theorem}[Volume-fixing variations]\label{thm volume fixing}
  If $\E_0$ is a $N$-cluster in $\R^n$, then there exist positive constants $r_0$, $\e_0$, $R_0$ and $C_0$ (depending on $\E_0$) with the following property. Let $\E$ be a $N$-cluster in $\R^n$ with
  \begin{equation}
        \label{volumefix hp 2}
    \d(\E,\E_0)\le\e_0\,,
  \end{equation}
  and let  $\F$ be a $N$-clusters in $\R^n$ such that either
  \begin{eqnarray}
    \label{volumefix hp 1}
    \bigcup_{h=1}^N\F(h)\Delta\E(h)\cc B_{x,r_0}\,,\qquad\mbox{for some $x\in\R^n$}\,,
  \end{eqnarray}
  or
  \begin{eqnarray}
    \label{volumefix hp 1x}
    \d(\E,\F)\le\om_n\,r_0^n\,,\qquad \bigcup_{h=1}^N\F(h)\Delta\E(h)\subset\R^n\setminus B_R\,,
    \qquad
    \begin{split}&\mbox{if there exists $R>0$ s.t.}
    \\&\bigcup_{h=1}^N\E_0(h)\cc B_R\,.\end{split}
  \end{eqnarray}
  Then there exists a $N$-cluster $\F'$ such that
  \begin{eqnarray}\label{volumefix thesis 1}
    \bigcup_{h=1}^N\F'(h)\Delta\F(h)&\cc&\left\{\begin{split}
      &B_{R_0}\setminus\ov{B_{x,r_0}}\,,&\qquad\mbox{if \eqref{volumefix hp 1} holds}\,,
      \\
      &B_{R}\,,&\qquad\mbox{if \eqref{volumefix hp 1x} holds}\,,
    \end{split}
    \right .
    \\\label{volumefix thesis 2}
    \vol(\F')&=&\vol(\E)\,,
    \\\label{volumefix thesis 3}
    |P(\F')-P(\F)|&\le& C_0\,P(\E)\,|\vol(\F)-\vol(\E)|\,,
    \\\label{volumefix thesis 4}
    |\d(\F',\E)- \d(\F,\E)|&\le&C_0\,P(\E)\,|\vol(\F)-\vol(\E)|\,,
    \\
    \label{volumefix thesis V}
    \sum_{h=0}^N\int_{\F'(h)\Delta\F(h)}J&\le&C_0\,\|J\|_{L^\infty(B_{R_0})}\,P(\E)\,\,|\vol(\F)-\vol(\E)|\,,
  \end{eqnarray}
  for every Borel function $J:\R^n\to[0,\infty)$ which is locally bounded.
\end{theorem}

We now prove Theorem \ref{thm asymptotic stability} and Theorem \ref{thm selection principle part one} for a fixed uniquely minimizing isoperimetric cluster $\E_0$. Thanks to \cite[Theorem 29.1]{maggiBOOK}, there exists $R>0$ such that $\E_0(h)\cc B_R$ for every $h=1,...,N$. Moreover we shall use the obvious inequality
\begin{equation}
    \label{alpha lipx}
      |\a(\E)-\a(\F)|\le\d(\E,\F)\,,\qquad\mbox{for every $N$-clusters $\E$ and $\F$}\,.
  \end{equation}

\begin{proof}[Proof of Theorem \ref{thm asymptotic stability}] The argument has several points in common with \cite[Proof of Theorem 29.1]{maggiBOOK}. Arguing by contradiction, we assume the existence of $\eta_*>0$ and of a sequence $\{\E_k\}_{k\in\N}$ of $N$-clusters such that $\vol(\E_k)=\vol(\E_0)$ for every $k\in\N$ and
\[
\lim_{k\to\infty}P(\E_k)=P(\E_0)\,,\qquad \lim_{k\to\infty}\a(\E_k)=\eta_*\,.
\]
By arguing as in step one of the proof of \cite[Theorem 29.1]{maggiBOOK} we identify for each cluster $\E_k$ a suitable region (constructed as a union of balls of radius $S$, see the right-hand side of \eqref{azz1}) inside of which, in the spirit of Theorem \ref{thm volume fixing}, we can perform volume-fixing variations of $\E_k$ with uniform bounds in $k$. More precisely, there exist positive constants $\e_1$, $C_1$, and $S$, points $\{x_k(h)\}_{k\in\N}\subset\R^n$ ($1\le h\le N$), and $C^1$-maps $\Phi_k:((-\e_1,\e_1)^{N+1}\cap V)\times\R^n\to\R^n$, (here $V=\{\aa\in\R^{N+1}:\sum_{h=0}^N\aa(h)=0\}$) with the property that (up to extracting subsequences in $k$) $\Phi_k(\aa,\cdot)$ is a $C^1$-diffeomorphism on $\R^n$ for every $\aa\in(-\e_1,\e_1)^{N+1}\cap V$, and, moreover, for every $\aa\in (-\e_1,\e_1)^{N+1}\cap V$ and for every $\H^{n-1}$-rectifiable set $\S\subset\R^n$, it holds
\begin{eqnarray}
\label{azz1}
  \big\{x\in\R^n:\Phi_k(\aa,x)\ne x\big\}&\cc& \bigcup_{h=1}^N B(x_k(h),S)\,,
  \\
\label{azz2}
  \big|\Phi_k(\aa,\E_k(h))\big|&=&|\E_k(h)|+\aa(h)\,,
  \\
\label{azz3}
  \big|\H^{n-1}\big(\Phi_k(\aa,\S)\big)-\H^{n-1}(\S)\big|&\le& C_1\,\H^{n-1}(\S)\,|\aa|\,,
  \\
\label{azz4}
  \big|\Phi_k(\aa,\E_k(h))\Delta\E_k(h)\big|&\le& C_1\,P(\E_k(h))\,|\aa|\,.
\end{eqnarray}
Note that \eqref{azz4} is not mentioned in step one of the proof of \cite[Theorem 29.1]{maggiBOOK}, but that it can be easily achieved by exploiting \cite[Lemma B.2]{CiLeMaIC1}. At the same time, by arguing as in step two of the proof of \cite[Theorem 29.1]{maggiBOOK}, we see that there exist positive constants $\e_0$ and $L$ (depending on $\{\E_k\}_{k\in\N}$ only) such that for every $\eta<\e_0$, $k\in\N$, and $h=1,\dots,N$, we can find finitely many points $\{y_k(h,i)\}_{i=1}^{L_k(h)}\subset\R^n$ such that
\begin{eqnarray}
  \label{azz5}
  \Big|\E_k(h)\setminus\bigcup_{i=1}^{L_k(h)}B(y_k(h,i),2)\Big|<\frac{\eta}{N}\,,
  \qquad L_k(h)\le \frac{L}{\eta^n}\,.
\end{eqnarray}
Let us now consider the closed sets
\[
F_k=\bigcup_{h=1}^N\,\ov{B}(x_k(h),S)\cup\bigcup_{i=1}^{L_k(h)}\,\ov{B}(y_k(h,i),2)\,,\qquad k\in\N\,.
\]
Since, by \eqref{azz5},
\begin{equation}
  \label{azz0}
  \sum_{h=1}^N|\E_k(h)\setminus F_k|\le \eta\,,\qquad\forall k\in\N\,,
\end{equation}
the {\it truncation lemma} \cite[Lemma 29.12]{maggiBOOK} guarantees the existence of $r_0\in[0,7n\,\eta^{1/n}]$ such that,
if $I_\e(X)=\{x\in\R^n:\dist(x,X)<\e\}$ denotes the $\e$-neighborhood of $X\subset\R^n$, and if $\{\E'_k\}_{k\in\N}$ are the $N$-clusters defined by
$\E'_k(h)=\E_k(h)\cap I_{r_0}(F_k)$, $1\le h\le N$,
then
\begin{equation}
  \label{azz15}
  P(\E'_k)\le P(\E_k)-\frac{\d(\E'_k,\E_k)}{4\,\eta^{1/n}}\,.
\end{equation}
By \eqref{azz0} we have $\d(\E'_k,\E_k)\le \eta$, so that by \eqref{alpha lipx}
\begin{equation}
  \label{azz6}
  \a(\E_k')\ge\a(\E_k)-\eta\,,\qquad\forall k\in\N\,.
\end{equation}
If we set $\aa_k(h)=|\E_k(h)|-|\E'_k(h)|=|\E_k(h)\setminus I_{r_0}(F_k)|$ for $1\le h\le N$ and $\aa_k(0)=-\sum_{h=1}^N\aa_k(h)$,
and if we require $\eta\le\e_1$, then $\aa_k\in (-\e_1,\e_1)^{N+1}\cap V$ for every $k\in\N$. We may thus define a sequence of clusters $\{\E_k''\}_{k\in\N}$ by setting
\[
\E_k''(h)=\Phi_k(\aa_k,\E_k'(h))\,,\qquad 1\le h\le N\,.
\]
Let us notice that, by \eqref{azz1}, $\Phi_k(x)=x$ in an open neighborhood of $\R^n\setminus F_k$, so that, in fact, $\Phi_k(\aa_k,\E_k'(h))=\Phi_k(\aa_k,\E_k(h))\cap I_{r_0}(F_k)$. Therefore, by  \eqref{azz2}, \eqref{azz3}, \eqref{azz15}, and the definition of the $\aa_k$'s, much as in step two of the proof of \cite[Theorem 29.1]{maggiBOOK}, we obtain that
\begin{eqnarray}\label{azz7}
  \vol(\E_k'')&=&\vol(\E_k)=\vol(\E_0)\,,
  \\\label{azz8}
  P(\E_k'')&\le& P(\E_k)+\Big(4C_1P(\E_0)-\frac1{4\eta^{1/n}}\Big)\,\d(\E'_k,\E_k)\,;
\end{eqnarray}
moreover, this time taking into account \eqref{azz4}, and since $\d(\E_k,\E_k')\le\eta$, we find that
\begin{equation}
  \label{azz9}
  \d(\E_k'',\E_k)\le\eta+C_1\,P(\E_k)\,|\aa_k|\le C_2\,\eta\,,
\end{equation}
where $C_2$ is a constant depending on $\{\E_k\}_{k\in\N}$ only; in particular, by \eqref{azz9} and \eqref{alpha lipx}
\begin{equation}
  \label{azz10}
  \a(\E_k'')\ge\a(\E_k)-C_2\eta\ge\frac{\eta_*}2\,,
\end{equation}
provided $\eta$ is small enough; similarly, up to further decreasing the value of $\eta$, \eqref{azz8} gives us
\begin{eqnarray}\label{azz8bis}
  P(\E_k'')\le P(\E_k)\,,\qquad\forall k\in\N\,.
\end{eqnarray}
Summarizing, by taking into account \eqref{azz7}, \eqref{azz8bis}, and \eqref{azz10} we see that $\{\E_k''\}_{k\in\N}$ satisfies
\begin{equation}
  \label{azz11}
  \lim_{k\to\infty}P(\E_k'')=P(\E_0)\,,\qquad \liminf_{k\to\infty}\a(\E_k'')\ge\frac{\eta_*}2\,;
\end{equation}
moreover, by the definition of $\E_k'$ and $\E_k''$, and thanks to \eqref{azz1},  for every $k\in\N$ we find
\begin{equation*}
  \bigcup_{h=1}^N\E''_k(h)\cc G_k=\ov{I_{2\,r_0}(F_k)}\,,
\end{equation*}
where $G_k$ is a closed set with at most $L_0=L_0(n,N,L,\eta)$ connected components of diameter at most $S_0=S_0(S,r_0,L_0)$ with $r_0\le 7n\eta^{1/n}$. Clearly, the mutual distances between these connected components may tend to infinity or not: in any case we can find $\{z_k^j\}_{j=1}^{M}\subset\R^n$, $1\le M\le L_0$, such that for every $k\in\N$ and $1\le j_1<j_2\le M$ (if $M\ge 2$)
\begin{eqnarray*}
\bigcup_{h=1}^N\E''_k(h)\cc \bigcup_{j=1}^M B(z_k^j,S_0)\,,\qquad\lim_{k\to\infty}|z_k^{j_1}-z_k^{j_2}|=\infty\,.
\end{eqnarray*}
In particular, $\{B(z_k^j,S_0)\}_{j=1}^M$ is a disjoint family of balls if $M\ge 2$ and $k$ is large enough. Let us assume, as we may up to isometries, that $\a(\E_k'')=\d(\E_k'',\E_0)$. Up to relabeling the index $j$ and up to take $k$ large enough, by taking into account $\E_0(h)\cc B_R$ for every $h=1,..,N$, we may ensure that
\begin{eqnarray*}
\a(\E_k'')=\sum_{h=1}^N \Big|\Big(\E_k''(h)\Delta \E_0(h)\Big)\cap B(z_k^1,S_0)\Big|\,,\qquad 0=\sum_{h=1}^N\sum_{j=2}^M\Big|\E_0(h)\cap B(z_k^j,S_0)\Big|\,.
\end{eqnarray*}
(This implies, in particular, that $|z_k^1|\le R+S_0$.) Let us finally consider vectors $\{y_k^j\}_{j=2}^M$ such that the balls $\{B(z_k^j+y_k^j,S_0)\}_{j=2}^M$ lie at mutually positive distance at least $2\,(S_0+R)$ and at most $2\,(S_0+R)\,M$ one from each other and from $B(z_k^1,S_0)$, and define a sequence $\{\E_k'''\}_{k\in\N}$ so that, for $h=1,\dots,N$,
\begin{eqnarray*}
\E_k'''(h)\cap B(z_k^1,S_0)&=&\E_k''(h)\cap B(z_k^1,S_0)\,,
\\
\E_k'''(h)\cap B(z_k^j+y_k^j,S_0)&=&\big(\E_k''(h)\cap B(z_k^j,S_0)\big)+y_k^j\,,\qquad 2\le j\le M\,,\quad
\\
\E_k'''(h)\setminus\Big(B(z_k^1,S_0)\cup\bigcup_{j=2}^MB(z_k^j,S_0)\Big)&=&\emptyset\,.
\end{eqnarray*}
In this way, by construction of $y_k^j$ and since $\E_0\cc B_R$, it must be $\a(\E_k''')=\a(\E_k'')$ for every $k$ large enough, so that
$\liminf_{k\to\infty}\a(\E_k''')\ge\eta_*/2$. At the same time, there exists $Q$ depending on $S_0$, $R$, and $M$ only, such that $\E_k'''\subset B_{Q}$ for every $k\in\N$, so that by $\lim_{k\to\infty}P(\E_k''')=P(\E_0)$, $\vol(\E_k''')=\vol(\E_0)$, and by the standard compactness theorem \cite[Proposition 29.5]{maggiBOOK}, there exists a $N$-cluster $\E_*$ such that, up to extracting subsequences, $\d(\E_k''',\E_*)\to 0$ as $k\to\infty$. Therefore, it holds $\vol(\E_*)=\vol(\E_0)$, $P(\E_*)= P(\E_0)$, and $\a(\E_*)\ge \eta_*/2$, a contradiction to the unique minimality of $\E_0$.
\end{proof}

\begin{proof}[Proof of Theorem \ref{thm selection principle part one}] Let us consider a recovery sequence $\{\F_k\}_{k\in\N}\subset\M_0$ for $\k(\E_0)$, that is
  \begin{equation}
  \label{recovery sequence}
  \inf_{k\in\N}\a(\F_k)>0\,,\qquad \lim_{k\to\infty}\d(\F_k,\E_0)=0\,,\qquad \kappa(\E_0)=\lim_{k\to\infty}\frac{\de(\F_k)}{\a(\F_{k})^2}\,,
  \end{equation}
  and notice that, since $\k(\E_0)<\infty$, we have
    \begin{equation}
    \label{recovery sequence properties}
    \lim_{k\to\infty}\a(\F_k)=0\,,\qquad P(\F_k)=P(\E_0)+\k(\E_0)\,\a(\F_k)^2+o(\a(\F_k)^2)\,.
  \end{equation}
  Without loss of generality, we may assume that, for all $k\in\N$, and for $\b>0$ to be suitably chosen,
\begin{eqnarray}
\label{stime def & asim-1}
P(\F_{k}) \leq P(\E_{0}) + (\k(\E_{0})+1)\, \a(\F_{k})^{2}\,,\qquad
\a(\F_{k}) \leq\beta\,.
\end{eqnarray}
We claim that for every $k$ large enough there exists a minimizer $\E_k$ in the problem
  \begin{equation}
  \label{selection principle k}
  \g_{k}(\E_0)=\inf\big\{P(\E)+|\a(\E) - \a(\F_{k})|^{3/2}:\ \E\in\M_0\big\}\,,
  \end{equation}
and that
\begin{eqnarray}
     \label{varie su Ek-1}
\a(\E_{k})&\ge& \frac{\a(\F_k)}3 \,,\\
    \label{varie su Ek-2}
|\a(\E_{k})-\a(\F_{k})|&\leq&  (\k(\E_{0}) + 1)^{2/3} \, \a(\F_{k})^{4/3}\,,
 \\\label{selezionato limitato}
 \bigcup_{h=1}^N\E_k(h)&\subset& B_{R_0}\,,\qquad R_0=R+7n\beta^{1/n}\,,
 \\
 \label{varie su Ek-3}
 P(\E_k)&=&P(\E_0)+\k(\E_0)\,\a(\E_k)^2+o(\a(\E_k)^2)\,,\qquad\mbox{as $k\to\infty$}\,.
\end{eqnarray}
Indeed, given $k\in\N$, let $\{\E_{k,j}\}_{j\in\N}$ be a minimizing sequence in \eqref{selection principle k}. Since $\F_{k}$ is admissible in \eqref{selection principle k} and by \eqref{stime def & asim-1}, provided $\b$ is small enough, we may assume without loss of generality that
\begin{equation}\label{competitore1}
\left\{\begin{split}
  &P(\E_{k,j}) + |\a(\E_{k,j}) - \a(\F_{k})|^{3/2} \leq P(\F_{k})
  \\
  &P(\E_{k,j})\le P(\E_0)+1
\end{split}\right .\,,\qquad \forall k\,,j\in\N \,.
\end{equation}
By subtracting $P(\E_0)$ in this last inequality, by $P(\E_{k,j})\ge P(\E_0)$, and by \eqref{stime def & asim-1} we thus get
\begin{eqnarray}\label{asymkh1}
|\a(\E_{k,j}) - \a(\F_{k})|^{3/2} \leq(\k(\E_{0})+1)\,\a(\F_{k})^2\,,\qquad \forall k\,,j\in\N\,.
\end{eqnarray}
In particular, provided $\beta$ is small enough, we find
\begin{equation}
    \label{stima asim Ekh-1}
 \frac{\a(\F_{k})}2\le \a(\E_{k,j}) \leq \frac{3}{2}\,\a(\F_{k})\,,\qquad \forall k\,,j\in\N\,.
\end{equation}
We now construct new minimizing sequences $\{\tE_{k,j}\}_{j\in\N}$ for the variational problems \eqref{selection principle k}, with the property that, for some $k_0\in\N$
\begin{equation}
  \label{new mini}
  \bigcup_{h=1}^N\tE_{k,j}(h)\subset B_{R+7n\beta^{1/n}}\,,\qquad\forall j\in\N\,,k\ge k_0\,.
\end{equation}
Indeed, let us assume, as we may do up to isometries, that
\begin{equation}
  \label{antelope}
  \a(\E_{k,j}) =\d(\E_{k,j},\E_{0})\,,\qquad\forall j\,,k\in\N\,.
\end{equation}
For each $k\,,j\in\N$ and $r>0$, we consider the cluster $\E_{k,j}^r(h)=\E_{k,j}(h)\cap B_r$, and correspondingly define a decreasing function $\rho_{k,j}:(0,\infty)\to[0,\infty)$ by setting
 \begin{equation}
   \label{birillo}
    \rho_{k,j}(r)=\d(\E_{k,j},\E_{k,j}^r)=\sum_{h=1}^{N} |\E_{k,j}(h)\setminus B_{r}|\,,\qquad k\,,j\in\N\,,r>0\,.
 \end{equation}
By $\bigcup_{h=1}^N\E_0(h)\cc B_R$, \eqref{antelope}, \eqref{stima asim Ekh-1} and \eqref{stime def & asim-1} we find
\begin{equation}
  \label{rhobeta}
  \rho_{k,j}(R) \leq \frac{\d(\E_{k,j},\E_{0})}2 = \frac{\a(\E_{k,j})}2 \leq \frac{3}4\,\a(\F_{k})\leq \frac{3}4\,\beta\,.
\end{equation}
Thus, by \cite[Lemma 29.12]{maggiBOOK}, there exists $r=r_{k,j}\in[R,R+7n\beta^{1/n}]$ such that
\begin{equation}\label{stima trunc}
P(\E_{k,j}^{r}) \leq P(\E_{k,j}) - \frac{\rho_{k,j}(r)}{4\,\beta^{1/n}}\,,\qquad \forall j\,,k\in\N\,,
\end{equation}
where in order to simplify the notation we have set $\E_{k,j}^{r}=\E_{k,j}^{r_{k,j}}$. Now let $\e_0$, $r_0$, and $C$ be the constants associated with $\E_0$ by Theorem \ref{thm volume fixing}, which we want to apply with the choices $\E=\E_{k,j}$ and $\F=\E_{k,j}^r$. This is possible because by \eqref{antelope}, \eqref{stima asim Ekh-1}, and \eqref{stime def & asim-1}, and provided $\beta$ is small enough, we have $\d(\E_{k,j},\E_0)\le\e_0$, while at the same time $\d(\E_{k,j},\E_{k,j}^r)\le\rho_{k,j}(R)\le\beta\le\om_n\,r_0^n$ and $\E_{k,j}(h)\Delta\E_{k,j}^r(h)\subset\R^n\setminus B_R$, where $R>0$ is such that $\E_0(h)\cc B_R$ for every $h=1,...,N$. By Theorem \ref{thm volume fixing} we thus construct clusters $\tE_{k,j}$ such that
\begin{equation}\label{birillo2}
\begin{split}
  &\vol(\tE_{k,j})=\vol(\E_{k,j})=\vol(\E_0)\,,
  \\
  &|\d(\tE_{k,j},\E_{k,j})-\d(\E_{k,j}^r,\E_{k,j})|\le C\,P(\E_{k,j})\,\rho_{k,j}(r)\,,
  \\
  &P(\tE_{k,j})\le P(\E_{k,j}^r)+C\,P(\E_{k,j})\,\rho_{k,j}(r)\,.
\end{split}
\end{equation}
By \eqref{alpha lipx}, \eqref{birillo2}, \eqref{competitore1}, and \eqref{birillo}  we find
\begin{eqnarray}
\label{perim trunc}
P(\tE_{k,j})-P(\E_{k,j}^{r})+|\a(\tE_{k,j}) - \a(\E_{k,j})|\leq C_1\,\rho_{k,j}(r)\,,
\end{eqnarray}
for some constant $C_1$ depending on $\E_0$ only. By \eqref{stima trunc} and \eqref{perim trunc} we find
\begin{eqnarray}\nonumber
&&P(\tE_{k,j}) + |\a(\tE_{k,j}) - \a(\F_{k})|^{3/2}
\\
&\leq&
P(\E_{k,j})  + |\a(\tE_{k,j}) - \a(\F_{k})|^{3/2}- \Big(\frac{1}{4\beta^{1/n}} -C_{1}\Big)\rho_{k,j}(r)\,,
\label{trunc minimal}
\end{eqnarray}
where, again thanks to \eqref{perim trunc} we have
\begin{eqnarray}\label{busoraro}
  |\a(\tE_{k,j})-\a(\F_{k})|^{3/2}\le\Big(|\a(\E_{k,j})-\a(\F_{k})|+C_1\,\rho_{k,j}(r)\Big)^{3/2}\,.
\end{eqnarray}
If $|\a(\E_{k,j})-\a(\F_{k})|\ge C_1\, \rho_{k,j}(r)$, then, by noticing that $(1+a)^{3/2} \leq 1+2a$ for every $a\in[0,1]$,
\begin{eqnarray}\nonumber
  |\a(\tE_{k,j})-\a(\F_{k})|^{3/2}&\le&|\a(\E_{k,j})-\a(\F_{k})|^{3/2}
  \Big(1+\frac{2C_1\,\rho_{k,j}(r)}{|\a(\E_{k,j})-\a(\F_{k})|}\Big)
  \\\nonumber
  &\le&|\a(\E_{k,j})-\a(\F_{k})|^{3/2}+2\,C_1\,\sqrt{\a(\F_k)}\,\rho_{k,j}(r)
  \\\label{busoo1}
  &\le&|\a(\E_{k,j})-\a(\F_{k})|^{3/2}+C_2\,\sqrt{\beta}\,\rho_{k,j}(r)\,,
\end{eqnarray}
thanks to \eqref{stime def & asim-1}, and for a constant $C_2$ depending on $\E_0$ only; if $|\a(\E_{k,j})-\a(\F_{k})|\le C_1\, \rho_{k,j}(r)$, then by \eqref{busoraro}, and up to possibly increasing the value of $C_2$, we simply find
\begin{eqnarray}\label{busoo2}
  |\a(\tE_{k,j})-\a(\F_{k})|^{3/2}\le\big(2\,C_1\,\rho_{k,j}(r)\big)^{3/2}\le C_2\,\sqrt\beta\,\rho_{k,j}(r)\,,
\end{eqnarray}
where we have used again \eqref{rhobeta} and the fact that $\rho_{k,j}$ is decreasing. We finally combine \eqref{trunc minimal}, \eqref{busoo1}, and \eqref{busoo2}, to conclude that, if $\beta$ is suitably small (in terms of $C_1$, $C_2$ and $n$), then
\begin{eqnarray}\nonumber
&&P(\tE_{k,j}) + |\a(\tE_{k,j}) - \a(\F_{k})|^{3/2}
\\
&\leq&
\label{busoo78}
P(\E_{k,j})  + |\a(\E_{k,j}) - \a(\F_{k})|^{3/2}- \Big(\frac{1}{4\beta^{1/n}} -C_{1}-C_2\,\sqrt{\beta}\Big)\rho_{k,j}(r)
\\\label{busoo3}
&\leq&
P(\E_{k,j})  + |\a(\E_{k,j}) - \a(\F_{k})|^{3/2}\,.
\end{eqnarray}
By \eqref{busoo3} and \eqref{new mini}, for every $k\in\N$, we find that $\{\tE_{k,j}\}_{j\in\N}\subset\M_0$ is a minimizing sequence in \eqref{selection principle k}, uniformly bounded in space. By the Direct Method (see, e.g. \cite[Propositons 29.4 and 29.5]{maggiBOOK}), up to possibly extracting a subsequence in $j$, there exist minimizers $\E_k$ in \eqref{selection principle k} such that $\d(\tE_{k,j},\E_{k})\to 0$ as $j\to \infty$. If we denote by $C_3$ the positive constant appearing in front of $-\rho_{k,j}(r)$ in \eqref{busoo78}, then by \eqref{busoo78}, \eqref{competitore1}, and \eqref{recovery sequence properties}, we find
\begin{eqnarray}\label{buso781}
  &&P(\E_0)\le P(\tE_{k,j})+ |\a(\tE_{k,j}) - \a(\F_{k})|^{3/2}+C_3\,\rho_{k,j}(r)\le P(\F_k)
  \\\label{buso782}
  &&\hspace{1.1cm}= P(\E_0)+\k(\E_0)\,\a(\F_k)^2+o(\a(\F_k)^2)\,.
\end{eqnarray}
By subtracting $P(\E_0)$, we can thus find $k_0\in\N$ such that, if $k\ge k_0$, then
\[
\sup_{h\in\N}\rho_{k,j}(r)\le \frac{(\k(\E_0)+1)}{C_3}\,\a(\F_k)^2\le\frac{\a(\F_k)}{6\,C_1}\,,
\]
possibly up to further decreasing the value of $\beta$. Correspondingly, by \eqref{perim trunc} and by the lower bound in \eqref{stima asim Ekh-1}, we find that
\[
\a(\tE_{k,j})\ge\a(\E_{k,j})-\frac{\a(\F_k)}{6}\ge \frac{\a(\F_k)}3\,, \qquad\forall j\in\N\,,k\ge k_0\,,
\]
so that \eqref{varie su Ek-1} follows by letting $j\to\infty$ and by using \eqref{alpha lipx}. By a similar argument we see that \eqref{buso781} and \eqref{stime def & asim-1} give us
\begin{equation}\label{loco2}
|\a(\tE_{k,j}) - \a(\F_{k})|^{3/2} \leq (\k(\E_{0})+1)\,\a(\F_{k})^2\,,\qquad\forall j,k\in\N\,.
\end{equation}
Thus \eqref{varie su Ek-2} follows by letting $j\to\infty$ in \eqref{loco2}, while \eqref{selezionato limitato} follows by letting $j\to\infty$ in \eqref{new mini}. By \eqref{buso781} and \eqref{buso782} we also see that
\[
P(\tE_{k,j})=P(\E_0)+\k(\E_0)\,\a(\F_k)^2+o(\a(\F_k)^2)=P(\E_0)+\k(\E_0)\,\a(\E_k)^2+o(\a(\E_k)^2)\,,
\]
where $\a(\E_k)/\a(\F_k)\to 1$ as $k\to\infty$ thanks to \eqref{asymkh1} and $\d(\tE_{k,j},\E_{k})\to 0$ as $j\to \infty$. Since $\liminf_{j\to\infty}P(\tE_{k,j})\ge P(\E_k)\ge P(\E_0)$ we deduce \eqref{varie su Ek-3}. We have thus completed the proof of the existence of minimizers $\E_k$ in \eqref{selection principle k} satisfying \eqref{varie su Ek-1}--\eqref{varie su Ek-3}.

We now prove that \eqref{C-k} holds for $k\ge k_0$. Indeed, if $\F\in\M(\E_0)$, then by minimality of $\E_{k}$ in \eqref{selection principle k} we have
\begin{equation}
  \label{busogol}
  P(\E_k)+|\a(\E_k)-\a(\F_k)|^{3/2}\le P(\F)+|\a(\F)-\a(\F_k)|^{3/2}\,.
\end{equation}
Since $|a^{3/2}-b^{3/2}|\le (3/2)\sqrt{\max\{a,b\}}|b-a|$ for every $a,b\ge 0$, we easily find that
\begin{equation}
  \label{busoautogol}
  |\a(\F)-\a(\F_k)|^{3/2}-|\a(\E_k)-\a(\F_k)|^{3/2}\le \frac32\,\sqrt{\a(\F_k)}\,|\a(\E_k)-\a(\F)|\,.
\end{equation}
We thus prove \eqref{C-k} by combining \eqref{busogol}, \eqref{busoautogol}, \eqref{varie su Ek-1}, and \eqref{alpha lipx}. We are left to prove that each $\E_{k}$ is a $(\Lambda,r_{0})$-perimeter minimizer, for some constants depending on $\E_0$ only. Indeed, let $\e_0$, $r_0$, and $C$ be the constants associated to $\E_0$ by Theorem \ref{thm volume fixing}. By \eqref{varie su Ek-2} and \eqref{stime def & asim-1}, up to further decreasing the value of $\beta$, we may assume that $\a(\E_{k})\le\e_{0}$ for all $k\in\N$, so that, up to isometries, we may assume that $\a(\E_{k})=d(\E_{k},\E_{0})\le \e_0$ for every $k\in\N$. Now we choose $x\in \R^{n}$ and an $N$-cluster $\F$ such that $\F(h)\Delta\E_{k}(h)\cc B(x,r_{0})$ for $h=1,...,N$. By applying Theorem \ref{thm volume fixing} with $\E=\E_{k}$, and up to further decreasing the value of $\beta$ to entail $P(\E_{k})\le2\,P(\E_{0})$, we construct a cluster $\F'$ satisfying $\F'(h)\Delta\F(h)\cc\R^n\setminus\ov{B(x,r_0)}$, $\vol(\F')=\vol(\F)$ and
\[
\max\big\{|P(\F')-P(\F)|,|\d(\F',\E_{k})- \d(\F,\E_{k})|\big\}\le 2C\,P(\E_{0})\,|\vol(\F)-\vol(\E_k)|\,,
\]
By exploiting these properties and \eqref{C-k}, and since $|\vol(\F) - \vol(\E_k)| \leq \d(\F,\E_{k})$, we thus find
\begin{eqnarray*}
  P(\E_{k}) &\leq& P(\F') + 3\sqrt{\a(\E_k)}\, \d(\F',\E_{k})
  \\
  &\leq& P(\F) + 2C\,P(\E_{0})\,(1+3\sqrt{\a(\E_k)})\,|\vol(\F)-\vol(\E_{0})|+3\sqrt{\a(\E_k)}\, \d(\F,\E_{k})
  \\
  &\le& P(\F)+\Lambda\,\d(\F,\E_k)\,,
\end{eqnarray*}
for a suitable value of $\Lambda$ determined by $\E_0$ only.
\end{proof}

\bibliography{../references}
\bibliographystyle{is-alpha}

\end{document}

%% file: sdb.pstex_t
\begin{picture}(0,0)%
\includegraphics{sdb.eps}%
\end{picture}%
\setlength{\unitlength}{3947sp}%
\begingroup\makeatletter\ifx\SetFigFont\undefined%
\gdef\SetFigFont#1#2#3#4#5{%
  \reset@font\fontsize{#1}{#2pt}%
  \fontfamily{#3}\fontseries{#4}\fontshape{#5}%
  \selectfont}%
\fi\endgroup%
\begin{picture}(4950,2020)(702,-2169)
\put(4560,-1268){\makebox(0,0)[lb]{\smash{{\SetFigFont{10}{12.0}{\rmdefault}{\mddefault}{\updefault}{\color[rgb]{0,0,0}$\E_0(2)$}%
}}}}
\put(869,-1286){\makebox(0,0)[lb]{\smash{{\SetFigFont{10}{12.0}{\rmdefault}{\mddefault}{\updefault}{\color[rgb]{0,0,0}$\E_0(1)$}%
}}}}
\put(1849,-1290){\makebox(0,0)[lb]{\smash{{\SetFigFont{10}{12.0}{\rmdefault}{\mddefault}{\updefault}{\color[rgb]{0,0,0}$\E_0(2)$}%
}}}}
\put(3185,-1856){\makebox(0,0)[lb]{\smash{{\SetFigFont{10}{12.0}{\rmdefault}{\mddefault}{\updefault}{\color[rgb]{0,0,0}$\E_0(1)$}%
}}}}
\end{picture}%

%% file: pushing.pstex_t
\begin{picture}(0,0)%
\includegraphics{pushing.eps}%
\end{picture}%
\setlength{\unitlength}{3947sp}%
\begingroup\makeatletter\ifx\SetFigFont\undefined%
\gdef\SetFigFont#1#2#3#4#5{%
  \reset@font\fontsize{#1}{#2pt}%
  \fontfamily{#3}\fontseries{#4}\fontshape{#5}%
  \selectfont}%
\fi\endgroup%
\begin{picture}(3421,1175)(476,-1062)
\put(565,-33){\makebox(0,0)[lb]{\smash{{\SetFigFont{10}{12.0}{\rmdefault}{\mddefault}{\updefault}{\color[rgb]{0,0,0}$\E_0$}%
}}}}
\put(2420,-67){\makebox(0,0)[lb]{\smash{{\SetFigFont{10}{12.0}{\rmdefault}{\mddefault}{\updefault}{\color[rgb]{0,0,0}$\E_t$}%
}}}}
\end{picture}%

%% file: calotta.pstex_t
\begin{picture}(0,0)%
\includegraphics{calotta.eps}%
\end{picture}%
\setlength{\unitlength}{3947sp}%
\begingroup\makeatletter\ifx\SetFigFont\undefined%
\gdef\SetFigFont#1#2#3#4#5{%
  \reset@font\fontsize{#1}{#2pt}%
  \fontfamily{#3}\fontseries{#4}\fontshape{#5}%
  \selectfont}%
\fi\endgroup%
\begin{picture}(5932,1977)(881,-1854)
\put(6705,-775){\makebox(0,0)[lb]{\smash{{\SetFigFont{9}{10.8}{\rmdefault}{\mddefault}{\updefault}{\color[rgb]{0,0,0}$x_1$}%
}}}}
\put(3172,-742){\makebox(0,0)[lb]{\smash{{\SetFigFont{9}{10.8}{\rmdefault}{\mddefault}{\updefault}{\color[rgb]{0,0,0}$x_1$}%
}}}}
\put(2164,-639){\makebox(0,0)[lb]{\smash{{\SetFigFont{9}{10.8}{\rmdefault}{\mddefault}{\updefault}{\color[rgb]{0,0,0}$\theta$}%
}}}}
\put(5606,-651){\makebox(0,0)[lb]{\smash{{\SetFigFont{9}{10.8}{\rmdefault}{\mddefault}{\updefault}{\color[rgb]{0,0,0}$\theta$}%
}}}}
\put(5225,-332){\makebox(0,0)[lb]{\smash{{\SetFigFont{10}{12.0}{\rmdefault}{\mddefault}{\updefault}{\color[rgb]{0,0,0}$S(\theta,u)$}%
}}}}
\put(1794,-322){\makebox(0,0)[lb]{\smash{{\SetFigFont{10}{12.0}{\rmdefault}{\mddefault}{\updefault}{\color[rgb]{0,0,0}$S(\theta)$}%
}}}}
\put(2666,-28){\makebox(0,0)[lb]{\smash{{\SetFigFont{10}{12.0}{\rmdefault}{\mddefault}{\updefault}{\color[rgb]{0,0,0}$A(\theta)$}%
}}}}
\put(6184,-53){\makebox(0,0)[lb]{\smash{{\SetFigFont{10}{12.0}{\rmdefault}{\mddefault}{\updefault}{\color[rgb]{0,0,0}$A(\theta,u)$}%
}}}}
\end{picture}%

%% file: sdb2.pstex_t
\begin{picture}(0,0)%
\includegraphics{sdb2.eps}%
\end{picture}%
\setlength{\unitlength}{3947sp}%
\begingroup\makeatletter\ifx\SetFigFont\undefined%
\gdef\SetFigFont#1#2#3#4#5{%
  \reset@font\fontsize{#1}{#2pt}%
  \fontfamily{#3}\fontseries{#4}\fontshape{#5}%
  \selectfont}%
\fi\endgroup%
\begin{picture}(4671,2222)(380,-1930)
\put(1022,-776){\makebox(0,0)[lb]{\smash{{\SetFigFont{9}{10.8}{\rmdefault}{\mddefault}{\updefault}{\color[rgb]{0,0,0}$\theta_0$}%
}}}}
\put(3605,-695){\makebox(0,0)[lb]{\smash{{\SetFigFont{9}{10.8}{\rmdefault}{\mddefault}{\updefault}{\color[rgb]{0,0,0}$\theta_2$}%
}}}}
\put(2081,-694){\makebox(0,0)[lb]{\smash{{\SetFigFont{9}{10.8}{\rmdefault}{\mddefault}{\updefault}{\color[rgb]{0,0,0}$\theta_1$}%
}}}}
\put(654,-998){\makebox(0,0)[lb]{\smash{{\SetFigFont{10}{12.0}{\rmdefault}{\mddefault}{\updefault}{\color[rgb]{0,0,0}$P_0$}%
}}}}
\put(2077,-1004){\makebox(0,0)[lb]{\smash{{\SetFigFont{10}{12.0}{\rmdefault}{\mddefault}{\updefault}{\color[rgb]{0,0,0}$P_1$}%
}}}}
\put(3366,-1010){\makebox(0,0)[lb]{\smash{{\SetFigFont{10}{12.0}{\rmdefault}{\mddefault}{\updefault}{\color[rgb]{0,0,0}$P_2$}%
}}}}
\put(2479, 63){\makebox(0,0)[lb]{\smash{{\SetFigFont{10}{12.0}{\rmdefault}{\mddefault}{\updefault}{\color[rgb]{0,0,0}$S$}%
}}}}
\put(1388,-403){\makebox(0,0)[lb]{\smash{{\SetFigFont{9}{10.8}{\rmdefault}{\mddefault}{\updefault}{\color[rgb]{0,0,0}$r_0$}%
}}}}
\put(2338,-438){\makebox(0,0)[lb]{\smash{{\SetFigFont{9}{10.8}{\rmdefault}{\mddefault}{\updefault}{\color[rgb]{0,0,0}$r_1$}%
}}}}
\put(3022,-339){\makebox(0,0)[lb]{\smash{{\SetFigFont{9}{10.8}{\rmdefault}{\mddefault}{\updefault}{\color[rgb]{0,0,0}$r_2$}%
}}}}
\put(4799,-735){\makebox(0,0)[lb]{\smash{{\SetFigFont{9}{10.8}{\rmdefault}{\mddefault}{\updefault}{\color[rgb]{0,0,0}$x_1$}%
}}}}
\put(2016,-1392){\makebox(0,0)[lb]{\smash{{\SetFigFont{10}{12.0}{\rmdefault}{\mddefault}{\updefault}{\color[rgb]{0,0,0}$\E_0(1)$}%
}}}}
\put(3197,-1499){\makebox(0,0)[lb]{\smash{{\SetFigFont{10}{12.0}{\rmdefault}{\mddefault}{\updefault}{\color[rgb]{0,0,0}$\E_0(2)$}%
}}}}
\end{picture}%

%% file: sym.pstex_t
\begin{picture}(0,0)%
\includegraphics{sym.eps}%
\end{picture}%
\setlength{\unitlength}{3947sp}%
\begingroup\makeatletter\ifx\SetFigFont\undefined%
\gdef\SetFigFont#1#2#3#4#5{%
  \reset@font\fontsize{#1}{#2pt}%
  \fontfamily{#3}\fontseries{#4}\fontshape{#5}%
  \selectfont}%
\fi\endgroup%
\begin{picture}(3373,1467)(662,-2154)
\put(677,-1324){\makebox(0,0)[lb]{\smash{{\SetFigFont{10}{12.0}{\rmdefault}{\mddefault}{\updefault}{\color[rgb]{0,0,0}$r\sqrt{3}$}%
}}}}
\put(2147,-1096){\makebox(0,0)[lb]{\smash{{\SetFigFont{9}{10.8}{\rmdefault}{\mddefault}{\updefault}{\color[rgb]{0,0,0}$r$}%
}}}}
\put(4000,-1336){\makebox(0,0)[lb]{\smash{{\SetFigFont{9}{10.8}{\rmdefault}{\mddefault}{\updefault}{\color[rgb]{0,0,0}$x_1$}%
}}}}
\put(2888,-1098){\makebox(0,0)[lb]{\smash{{\SetFigFont{9}{10.8}{\rmdefault}{\mddefault}{\updefault}{\color[rgb]{0,0,0}$r$}%
}}}}
\put(2072,-1605){\makebox(0,0)[lb]{\smash{{\SetFigFont{10}{12.0}{\rmdefault}{\mddefault}{\updefault}{\color[rgb]{0,0,0}$P_1$}%
}}}}
\put(2947,-1599){\makebox(0,0)[lb]{\smash{{\SetFigFont{10}{12.0}{\rmdefault}{\mddefault}{\updefault}{\color[rgb]{0,0,0}$P_2$}%
}}}}
\put(1900,-1971){\makebox(0,0)[lb]{\smash{{\SetFigFont{10}{12.0}{\rmdefault}{\mddefault}{\updefault}{\color[rgb]{0,0,0}$\E_0(1)$}%
}}}}
\put(2727,-1983){\makebox(0,0)[lb]{\smash{{\SetFigFont{10}{12.0}{\rmdefault}{\mddefault}{\updefault}{\color[rgb]{0,0,0}$\E_0(2)$}%
}}}}
\put(3222,-1273){\makebox(0,0)[lb]{\smash{{\SetFigFont{9}{10.8}{\rmdefault}{\mddefault}{\updefault}{\color[rgb]{0,0,0}$2\pi/3$}%
}}}}
\end{picture}%

%% file: angoli.pstex_t
\begin{picture}(0,0)%
\includegraphics{angoli.eps}%
\end{picture}%
\setlength{\unitlength}{3947sp}%
\begingroup\makeatletter\ifx\SetFigFont\undefined%
\gdef\SetFigFont#1#2#3#4#5{%
  \reset@font\fontsize{#1}{#2pt}%
  \fontfamily{#3}\fontseries{#4}\fontshape{#5}%
  \selectfont}%
\fi\endgroup%
\begin{picture}(3763,1072)(380,-1073)
\put(654,-998){\makebox(0,0)[lb]{\smash{{\SetFigFont{10}{12.0}{\rmdefault}{\mddefault}{\updefault}{\color[rgb]{0,0,0}$P_0=(0,0)$}%
}}}}
\put(2077,-1004){\makebox(0,0)[lb]{\smash{{\SetFigFont{10}{12.0}{\rmdefault}{\mddefault}{\updefault}{\color[rgb]{0,0,0}$P_1=(t_1,0)$}%
}}}}
\put(3366,-1010){\makebox(0,0)[lb]{\smash{{\SetFigFont{10}{12.0}{\rmdefault}{\mddefault}{\updefault}{\color[rgb]{0,0,0}$P_2=(t_2,0)$}%
}}}}
\put(1388,-403){\makebox(0,0)[lb]{\smash{{\SetFigFont{9}{10.8}{\rmdefault}{\mddefault}{\updefault}{\color[rgb]{0,0,0}$r_0$}%
}}}}
\put(1022,-776){\makebox(0,0)[lb]{\smash{{\SetFigFont{9}{10.8}{\rmdefault}{\mddefault}{\updefault}{\color[rgb]{0,0,0}$\theta_0$}%
}}}}
\put(2199,-596){\makebox(0,0)[lb]{\smash{{\SetFigFont{9}{10.8}{\rmdefault}{\mddefault}{\updefault}{\color[rgb]{0,0,0}$r_1$}%
}}}}
\put(3185,-502){\makebox(0,0)[lb]{\smash{{\SetFigFont{9}{10.8}{\rmdefault}{\mddefault}{\updefault}{\color[rgb]{0,0,0}$r_2$}%
}}}}
\put(2601,-395){\makebox(0,0)[lb]{\smash{{\SetFigFont{9}{10.8}{\rmdefault}{\mddefault}{\updefault}{\color[rgb]{0,0,0}$\pi/3$}%
}}}}
\put(2175,-350){\makebox(0,0)[lb]{\smash{{\SetFigFont{9}{10.8}{\rmdefault}{\mddefault}{\updefault}{\color[rgb]{0,0,0}$\pi/3$}%
}}}}
\end{picture}%

%% file: PlanarDB_submit.bbl
\newcommand{\etalchar}[1]{$^{#1}$}
\begin{thebibliography}{HMRR02}

\bibitem[Alm76]{Almgren76}
F.~J.~Jr. Almgren.
\newblock Existence and regularity almost everywhere of solutions to elliptic
  variational problems with constraints.
\newblock {\em Mem. Amer. Math. Soc.}, 4\penalty0 (165):\penalty0 viii+199 pp,
  1976.

\bibitem[BBJ14]{barchiesibrancolinijulin}
A.~Brancolini, M.~Barchiesi, and V.~Julin.
\newblock Sharp dimension free quantitative estimates for the {G}aussian
  isoperimetric inequality.
\newblock 2014.
\newblock http://cvgmt.sns.it/paper/2516/.

\bibitem[BDF12]{bogelainduzaarfusco}
V.~B\"ogelein, F.~Duzaar, and N.~Fusco.
\newblock A sharp quantitative isoperimetric inequality in higher codimension.
\newblock 2012.
\newblock http://cvgmt.sns.it/paper/1865/.

\bibitem[Ber05]{bernstein}
F.~Bernstein.
\newblock Uber die isoperimetriche eigenschaft des kreises auf der
  kugeloberflache und in der ebene.
\newblock {\em Math. Ann.}, 60:\penalty0 117--136, 1905.

\bibitem[Bon24]{bonnesen}
T.~Bonnesen.
\newblock Uber die isoperimetrische defizit ebener figuren.
\newblock {\em Math. Ann.}, 91:\penalty0 252--268, 1924.

\bibitem[CFMP11]{cianchifuscomaggipratelliGAUSS}
A.~Cianchi, N.~Fusco, F.~Maggi, and A.~Pratelli.
\newblock On the isoperimetric deficit in gauss space.
\newblock {\em Amer. J. Math.}, 133\penalty0 (1):\penalty0 131--186, 2011.

\bibitem[CL12]{CicaleseLeonardi}
M.~Cicalese and G.~P. Leonardi.
\newblock A selection principle for the sharp quantitative isoperimetric
  inequality.
\newblock {\em Arch. Rat. Mech. Anal.}, 206\penalty0 (2):\penalty0 617--643,
  2012.

\bibitem[CLM14]{CiLeMaIC1}
M.~Cicalese, G.~P. Leonardi, and F.~Maggi.
\newblock Improved convergence theorems for bubble clusters. {I}. {T}he planar
  case.
\newblock preprint arXiv:1409.6652, 2014.

\bibitem[DPM14]{dephilippismaggi}
G.~De~Philippis and F.~Maggi.
\newblock Sharp stability inequalities for the {P}lateau problem.
\newblock {\em J. Differential Geom.}, 96\penalty0 (3):\penalty0 399--456,
  2014.

\bibitem[FAB{\etalchar{+}}93]{small_doublebubble}
J.~Foisy, M.~Alfaro, J.~Brock, N.~Hodges, and J.~Zimba.
\newblock The standard double soap bubble in $\mathbb{R}^2$ uniquely minimizes
  perimeter.
\newblock {\em Pacific J. Math.}, 159\penalty0 (1), 1993.

\bibitem[FGP12]{fuscogellipisante}
N.~Fusco, M.~Gelli, and G.~Pisante.
\newblock On a {B}onnesen type inequality involving the spherical deviation.
\newblock {\em J. Math. Pures Appl. (9)}, 98\penalty0 (6):\penalty0 616--632,
  2012.

\bibitem[FJ14]{fuscojulin}
N.~Fusco and V.~Julin.
\newblock A strong form of the quantitative isoperimetric inequality.
\newblock {\em Calc. Var. Partial Differential Equations}, 50\penalty0
  (3--4):\penalty0 925--937, 2014.

\bibitem[FM11]{FigalliMaggiARMA}
A.~Figalli and F.~Maggi.
\newblock On the shape of liquid drops and crystals in the small mass regime.
\newblock {\em Arch. Rat. Mech. Anal.}, 201:\penalty0 143--207, 2011.

\bibitem[FMM11]{fuscomillotmorini}
N.~Fusco, V.~Millot, and M.~Morini.
\newblock A quantitative isoperimetric inequality for fractional perimeters.
\newblock {\em J. Funct. Anal.}, 261\penalty0 (3):\penalty0 697--715, 2011.

\bibitem[FMP08]{fuscomaggipratelli}
N.~Fusco, F.~Maggi, and A.~Pratelli.
\newblock The sharp quantitative isoperimetric inequality.
\newblock {\em Ann. Math.}, 168:\penalty0 941--980, 2008.

\bibitem[FMP10]{FigalliMaggiPratelliINVENTIONES}
A~Figalli, F.~Maggi, and A.~Pratelli.
\newblock A mass transportation approach to quantitative isoperimetric
  inequalities.
\newblock {\em Inv. Math.}, 182\penalty0 (1):\penalty0 167--211, 2010.

\bibitem[Fug89]{Fuglede}
B.~Fuglede.
\newblock Stability in the isoperimetric problem for convex or nearly spherical
  domains in $\mathbb{R}^n$.
\newblock {\em Trans. Amer. Math. Soc.}, 314:\penalty0 619--638, 1989.

\bibitem[Fug93]{Fuglede93}
B.~Fuglede.
\newblock Lower estimate of the isoperimetric deficit of convex domains in
  $\mathbb{R}^n$ in terms of asymmetry.
\newblock {\em Geom. Dedicata}, 47\penalty0 (1):\penalty0 41--48, 1993.

\bibitem[Hal92]{hall}
R.~R. Hall.
\newblock A quantitative isoperimetric inequality in $n$-dimensional space.
\newblock {\em J. Reine Angew. Math.}, 428:\penalty0 161--176, 1992.

\bibitem[HHW91]{hallhaymanweitsman}
R.~R. Hall, W.~K. Hayman, and A.~W. Weitsman.
\newblock On asymmetry and capacity.
\newblock {\em J. d'Analyse Math.}, 56:\penalty0 87--123, 1991.

\bibitem[HMRR02]{HutchMorganRosRitore}
M.~Hutchings, F.~Morgan, M.~Ritor\'e, and A.~Ros.
\newblock Proof of the double bubble conjecture.
\newblock {\em Ann. of Math. (2)}, 155\penalty0 (2):\penalty0 459--489, 2002.

\bibitem[Mag08]{maggibams}
F.~Maggi.
\newblock Some methods for studying stability in isoperimetric type problems.
\newblock {\em Bull. Amer. Math. Soc.}, 45:\penalty0 367--408, 2008.

\bibitem[Mag12]{maggiBOOK}
F.~Maggi.
\newblock {\em Sets of finite perimeter and geometric variational problems: an
  introduction to Geometric Measure Theory}, volume 135 of {\em Cambridge
  Studies in Advanced Mathematics}.
\newblock Cambridge University Press, 2012.

\bibitem[MN15]{mossel}
E.~Mossel and J.~Neeman.
\newblock Robust optimality of {G}aussian noise stability.
\newblock {\em J. Eur. Math. Soc. (JEMS)}, 17\penalty0 (2):\penalty0 433--482,
  2015.

\bibitem[Rei08]{Reichardt}
B.~W. Reichardt.
\newblock Proof of the double bubble conjecture in $\mathbb{R}^n$.
\newblock {\em J. Geom. Anal.}, 18\penalty0 (1):\penalty0 172--191, 2008.

\bibitem[RW13]{renwei}
X.~Ren and J.~Wei.
\newblock A double bubble in a ternary system with inhibitory long range
  interaction.
\newblock {\em Arch. Ration. Mech. Anal.}, 208\penalty0 (1):\penalty0 201--253,
  2013.

\bibitem[Wic04]{wichi}
W.~Wichiramala.
\newblock Proof of the planar triple bubble conjecture.
\newblock {\em J. Reine Angew. Math.}, 567:\penalty0 1--49, 2004.

\end{thebibliography}
